%% file: loops.tex
\documentclass[11pt]{article}
\usepackage[utf8]{inputenc}
\usepackage{amsfonts, amsmath}
\usepackage{amssymb}
\usepackage{amsthm,amsmath,amscd}
\usepackage{enumerate}
\usepackage{url}
\usepackage{amsbsy}
\usepackage[pdftex]{graphicx}
\usepackage[margin=25pt,font=small,labelfont=bf]{caption}
\usepackage[breaklinks,hyperfootnotes=false,bookmarks=false,colorlinks=true]{hyperref}
\usepackage{subfig}
\usepackage{mdframed}
\usepackage{color}
\usepackage[square,numbers,sort&compress]{natbib}
\newtheorem{innercustomthm}{Theorem}
\newenvironment{customthm}[1]
  {\renewcommand\theinnercustomthm{#1}\innercustomthm}
  {\endinnercustomthm}

\newcommand{\defn}[1]{\textcolor{blue}{\emph{#1}}}

\graphicspath{{figs/}}

\DeclareMathOperator{\sing}{Sing}

\newcommand{\RR}{\mathbb R}

\newcommand{\EE}{\mathbb E}
\newcommand{\R}{\mathbb R}

\newcommand{\bna}{\begin{eqnarray}}
\newcommand{\ena}{\end{eqnarray}}
\newcommand{\ba}{\begin{eqnarray*}}
\newcommand{\ea}{\end{eqnarray*}}
\newcommand{\bs}[1]{}

\newcommand{\edgecard}{N}

\newtheorem{theorem}{Theorem}[section]

\newtheorem{lemma}[theorem]{Lemma}
\newtheorem{proposition}[theorem]{Proposition}
\newtheorem{remark}[theorem]{Remark}
\newtheorem{definition}[theorem]{Definition}

\DeclareMathOperator{\real}{Real}

\DeclareMathOperator{\Dim}{Dim}

\newcommand{\ra}{\rangle}
\newcommand{\la}{\langle}
\newcommand{\CC}{{\mathbb C}}
\newcommand{\CS}{{\mathcal S}}
\newcommand{\QQ}{{\mathbb Q}}
\newcommand{\QQB}{{\overline{\mathbb Q}}}

\newcommand{\bl}{{\bf l}}

\newcommand{\LL}{{L_{2,4}}}

\newcommand{\LN}{{L_{d,n}}}
\def\p{{\bf p}}

\def\pn{\p =(\p_1, \dots, \p_{n}) }
\def\q{{\bf q}}

\def\bm{{\bf m}}
\def\v{{\bf v}}
\def\w{{\bf w}}
\def\e{{\bf e}}
\def\r{{\bf r}}
\def\E{{\bf E}}

\def\x{{\bf x}}

\def\M{{\bf M}}
\def\I{{\bf I}}

\def\N{{\bf N}}
\def\A{{\bf A}}
\def\B{{\bf B}}

\def\S{{\bf S}}

\def\K{{\bf K}}

\def\P{{\bf P}}

\newcommand{\genericpoint}{\bl}

\title{
Trilateration using  Unlabeled 
Path or Loop Lengths}

\author{
Ioannis Gkioulekas, 
Steven J. Gortler,
Louis Theran,
and 
Todd Zickler}
\date{}

\usepackage[margin=1in]{geometry}
\begin{document}
\maketitle 

\begin{abstract}
Let $\p$ be a configuration of $n$ points in $\RR^d$ for some 
$n$ and some
$d \ge 2$. Each pair of points defines an edge, which has a Euclidean
length in the configuration. A path is an ordered sequence of the
points, and a loop is a path that begins and ends at the same point. A path or loop, as a sequence of edges, also has a Euclidean length, which is simply the sum of its Euclidean edge lengths.
We are interested in reconstructing $\p$ given a set of edge, path and loop lengths.
In particular, we consider the unlabeled setting where the lengths are
given simply as a set of real numbers, and are 
not labeled with the combinatorial data describing which paths 
or loops gave rise to these lengths.
In this paper, 
we study the question of when $\p$ will be uniquely determined (up to an
unknowable Euclidean transform) from some given set of path or loop 
lengths through an exhaustive trilateration process. 

Such a process has  already been used for the simpler problem
of reconstruction using unlabeled edge lengths. This paper also provides a 
complete proof that this process must work in that edge-setting when
given a sufficiently rich set of edge measurements and assuming that $\p$
is generic.
\end{abstract}

%%%%%%%%%%%%%%%%
%\newpage
%%%%%%%%%%%%%%%%
\section{Introduction} 
\label{sec:intro}

%\subsection{Setup and motivation}
We are motivated by the following signal processing
scenario.
Suppose there is a ``configuration'' $\pn$ of $n$ points 
in, say, $\RR^2$ or $\RR^3$. Let a ``path'' be a finite sequence of these points, 
and a ``loop'' be a path that begins and ends at the same point. (We will use these terms to refer to the formal graph-theoretic notions of a ``walk'' and a ``closed walk''.) Each such path or loop in $\p$ has a Euclidean length.

\begin{figure}[b]
	\centering
	\def\svgwidth{4in}
	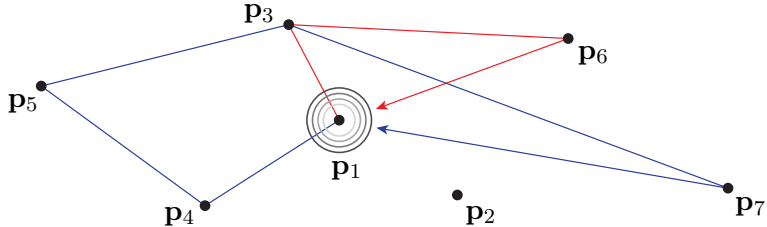
	\caption{An emitter-receiver at point $\p_1$ emits an omnidirectional pulse that bounces among points $\p_i$. The same emitter-receiver records the arrival times of pulse fronts that eventually return. These arrival times measure the Euclidean lengths of loops that begin and end at $\p_1$. 
 %They do not include any labeling information about the number of bounces or the points that were visited in each loop.
}\label{fig:intro}
\end{figure}

Let $\p_1$ be a distinguished point. In our scenario, 
it may represent the location of an omnidirectional emitter and receiver of sound or radiation. Let the other points in $\p$ represent the positions of small objects that behave as omnidirectional scatterers.

An omnidirectional pulse is emitted from $\p_1$ and travels outward at, say, unit speed. Whenever the pulse front encounters an object $\p_i$, an additional omnidirectional pulse is created there through scattering. Pulses continue to bounce around in this manner, and the receiver at $\p_1$ records the arrival times of the pulse fronts that return. We allow for the possibility that some pulse fronts might vanish or not be measurable back at $\p_1$.

By recording the times of flight between emission and reception, we effectively measure the lengths of loops traveled. In the case of light, these are travel times of photons that leave $\p_1$ and return after one or more bounces. In the case of sound, these are delays of direct or indirect echoes.

Each recorded length measurement is a single real number $v$. Importantly, we do not obtain any \emph{labeling} information about which points were visited or how many bounces occurred during the loop. 
We also do not obtain any information about the direction from which energy arrives. 

We wish to understand if we can recover the point configuration (up to Euclidean congruence) from a sufficiently rich sequence of \emph{unlabeled} loop measurements.  Once the loop measurements are labeled, the 
reconstruction problem is closely related to the well-studied 
graph realization (or ``distance geometry'') problem \cite{survey}.
Various techniques work well in practice for sufficiently rich inputs.
The primary difficulty in our settings arises from the lack of labeling.

Having no combinatorial information appears, at first sight, 
to be a very daunting problem.
Our first insight is to apply the notion of trilateration, which has been proposed as a method to recover a molecular shape from unlabeled inter-atomic distances~\cite{dux2}.  
As a side
contribution of this paper,
we provide a complete
proof that ``unlabeled trilateration'' 
must work given 
a sufficiently rich set of such distances as well as 
a genericity assumption on $\p$.

In our scenario, trilateration lets us decompose our labeling problem into a sequence of smaller ones, but in applying this technique to our problem we find a crucial distinction from the molecular one. In molecular applications,
one relies on the assumption that each measurement arises due to some atomic pair in the configuration (represented as an edge between two points). The story changes if
we additionally allow for the possibility of 
path-length measurements to be mixed in with edge-length measurements, because path-length measurements could perhaps be incorrectly interpreted as edge-lengths and thus confuse such a method. 
Our second insight is that we can characterize the effect of such additional path-length data by studying a novel algebraic variety that we call the ``unsquared measurement variety''.
Briefly, path-data that we could confuse for edge-data must arise 
from a linear automorphism of this variety.  Since 
we have classified such linear automorphisms~\cite{loopsAlg}, we can understand
the effect that extra path-length data has on our ability to recover a configuration.

Using these insights, 
we will prove in this paper that if $\p$ is a ``generic'' 
point configuration in $\R^d$ for $d\ge 2$, and we measure the lengths of a
sufficiently rich set of loops, namely
one that ``allows for trilateration'' (formally defined later),
then the configuration $\p$ is uniquely
determined from these measurements up to congruence.  Moreover, this
leads to an algorithm, under a real computation model~\cite{bss},
to calculate $\p$ from such data.  
The assumption of genericity (defined later) 
roughly means that while there are some special $\p$ where
these conclusions do not hold, these special cases
are very rare. In deriving our results, we will not concern ourselves with 
noise or numerical issues. We
plan to address some of these issues in future work.

\subsection*{Acknowledgements}
We would like to thank Dylan Thurston
for numerous helpful
conversations and suggestions throughout this project. 
We thank Brian Osserman with his help on algebraic geometry questions.

Steven Gortler was partially supported by NSF grant DMS-1564473. Ioannis Gkioulekas and Todd Zickler received support from the DARPA REVEAL program under contract no.~HR0011-16-C-0028.

\section{Idea Overview}
\label{sec:introa}

\subsection{Edge measurements}
To put this work in the context of previous mathematical results,
let us begin with the simpler setting, where we are given
an unlabeled set of edge-lengths.

\subsubsection{First case: complete graphs $K_n$}
Boutin and Kemper~\cite{BK1} 
(Theorem~\ref{thm:bkMain} below)
have shown that if $\p$ is a generic
$n$-point configuration in $\R^d$ with $n \geq d+2$, and we are given the complete
set of all $N:=\binom{n}{2}$
edge lengths as an unlabeled set, then
$\p$ is uniquely determined up to Euclidean congruence and point
relabeling from this data.  Since, in this case, we have all possible $N$ edge measurements, 
we can associate these measurements with the edges of the complete graph $K_n$.

The main idea behind the Boutin--Kemper result is to study the linear automorphisms
of the squared measurement variety $M_{d,n} \subseteq \CC^N $,
defined below,
of $n$ points
in $d$ dimensions. The variety $M_{d,n}$ represents all of the possible $N$-sets of 
squared edge-length measurements over all possible configurations
of $n$ points in $d$ dimensions. (For technical reasons,
such varieties  are most easily
studied in the complex setting.) Boutin and Kemper show
(Theorem~\ref{thm:bk-linear-s} below) that if an ``edge
permutation'' (permutation of the coordinate axes of $\CC^N$) 
gives rise to a linear  automorphism of $M_{d,n}$
(maps the variety to itself), then this edge permutation must
arise due to a relabeling of the $n$ vertices.

Theorem~\ref{thm:bkMain} then follows from 
Theorem~\ref{thm:bk-linear-s} and the following general principle
(see Appendix~\ref{sec:geometry} for a proof).
\begin{theorem}
\label{thm:prin2}
Let $V \subseteq \CC^N$ be an irreducible algebraic variety
and
$\genericpoint$ a generic point of $V$.
Let $\A$ be a bijective linear map on $\CC^N$
that maps $\genericpoint$
to  a point in $V$. Let all the above be defined over
$\QQ$. Then $\A(V) = V$, that is, $\A$ 
acts as a linear automorphism of $V$.
\end{theorem}
In the intended application, 
$\genericpoint \in \CC^N$ is the correctly ordered 
set of edge lengths. Since it represents a consistent set
of edge lengths from $K_n$, it is a point in $V:=M_{d,n}$.
The linear map $\A$ represents some potential edge permutation.
If the permuted squared lengths still fit together consistently as an $n$-point
configuration, then $\A(\genericpoint)$
is also in $M_{d,n}$ and 
from Theorem \ref{thm:prin2} $\A$ must act as a linear  automorphism of $M_{d,n}$.
Then from Theorem~\ref{thm:bk-linear-s}, we can conclude that this permutation can only arise
from a vertex relabeling. Any other type of permutation must place 
$A(\genericpoint)$, with $\genericpoint$ generic,  outside of $M_{d,n}$.

We can turn Boutin--Kemper's theorem into an algorithm for 
reconstructing a
generic $\p$ from its unordered edge length 
measurements.   
To check whether an ordering $\genericpoint = (l_{ij})$ of the edge measurements 
is in the complex variety, $M_{d,n}$, we compute the rank of the $(n-1)\times (n-1)$ Gram matrix:
\[
	\begin{pmatrix}
		2l^2_{1n} & \cdots & \cdots & l^2_{1n} + l^2_{(n-1)n} - l^2_{1(n-1)} \\
		\vdots & \ddots & &  \vdots \\
		\vdots &  & \ddots & \vdots \\
		l^2_{1n} + l^2_{(n-1)n} - l^2_{1(n-1)} & \cdots  & \cdots  & 2l^2_{(n-1)n}
	\end{pmatrix}
\]
If this rank is at most $d$, then $\genericpoint\in M_{d,n}$.  
Due to the genericity assumption, if $\genericpoint$ is \emph{not}
correctly ordered, then 
$\genericpoint$ will \emph{not} be 
in $M_{d,n}$ and the matrix will have a larger rank.
As such, no explicit 
positive semidefiniteness test is needed for this step; see Remark~\ref{rem:noPsdTest}.
Algorithmically, 
we can simply try different orderings for 
$\genericpoint$  until we find 
one that is in $M_{d,n}$.  Such an ordering will exist due to the assumption
that this data arose from an actual configuration $\p$.
From the assumed genericity, this ordering will be unique
(up to vertex relabeling).
So this ordering must correspond
to the ordered complete graph that was used to measure $\p$.
Next, since $\p$ was real, the Gram matrix must be positive semidefinite (PSD). Thus, it can be factored to find the real-valued configuration
\cite{schoen,YH38}. 

This algorithm is only applicable in practice for very small $n$, but a 
more efficient approach is based on applying it iteratively to smaller subsets of the data,
as described below.  The overall approach of generating candidate 
combinatorial types and testing them with polynomial predicates 
will appear throughout what follows.

\subsubsection{Second case: $K_{d+2}$ subconfigurations}

Next, suppose that 
$\p$ is a generic $n$-point 
configuration in $\R^d$ with $n \geq d+2$, and we are given
a set 
of $D:=\binom{d+2}{2}$
edge lengths as an unlabeled sequence.
Suppose that these $D$ lengths are 
consistent with 
the measured edge set over one   small complete graph,
$K_{d+2}$. 
We can show (Proposition~\ref{prop:ind} below) that these lengths
must actually have arisen through 
the edges of a $K_{d+2}$.
To do this we will invoke the following general result 
(see Appendix~\ref{sec:geometry} for a proof) which generalizes Theorem \ref{thm:prin2}:
\begin{theorem}
\label{thm:prin1}
Let $V$ be an irreducible algebraic variety
and
$\genericpoint$ a generic point of $V$.
Let $\E$ be a linear map that maps $\genericpoint$
to a point in some variety $W$. Let all the above be defined over
$\QQ$. Then $\E(V) \subseteq W$.
\end{theorem}
In our application, $\E$ will be the linear map that projects $M_{d,n}$
onto a specific set of coordinates corresponding to $D$ edges, 
and $W$ will be $M_{d,d+2}$.  
A rigidity-theoretic argument
tells us that if the 
$D$ edges do  \emph{not} form a $K_{d+2}$ subgraph of $K_n$, the
image $\E(M_{d,n})$ will be $D$-dimensional.  On the other hand, 
$M_{d,d+2}$ has dimension $D-1$.  From  Theorem \ref{thm:prin1}, 
we then see that if $\bl$ is generic and 
$\E(\bl)$ is in $M_{d,d+2}$,
then $\E$ must correspond to a $K_{d+2}$, since this is the only 
way for $\E(M_{d,n})$ to have dimension less than $D$.

As in the previous section, we can test whether $D$ ordered measurements
are in $M_{d,d+2}$ by a matrix rank computation.  
(No PSD test is needed here,
as described in Remark~\ref{rem:consistPSD}.)
If they are in $M_{d,d+2}$, we can then
apply Boutin and Kemper's result to this subset, 
and uniquely reconstruct the associated subconfiguration
of these $d+2$ points, up to congruence.

\subsubsection{Trilateration}

This idea can now be used (when $d\ge 2)$
to reconstruct all $n$ of the points
under the assumption that our unlabeled data includes the measurements of
an edge subset that is rich enough to allow for trilateration~\cite{dux2}. 
Loosely speaking, this means that
the measured edge set contains one   complete graph
$K_{d+2}$, and then includes more edges that allow us to 
inductively glue all of the vertices, one by one,  onto the 
currently reconstructed point set. 
Each such inductive step involves one new vertex $v$ with $d+1$
edge measurements connecting $v$ to the already reconstructed
set. This essentially allows us to find another $K_{d+2}$
graph that overlaps sufficiently with the already reconstructed
point set so that they can be glued together
in a unique manner. (See Figure~\ref{fig:pt}, top.)
As argued above, the geometry of each such $K_{d+2}$ is completely determined by
its unlabeled edge lengths.

Assuming our full data set allows for trilateration, then there is
a unique $n$-point configuration consistent with the data, up to congruence.
Finding these $K_{d+2}$ subgraphs
requires a somewhat exhaustive search
over the data set, giving us a running time that is 
exponential in $d$ (which we think of as fixed) but only 
polynomial in $n$.

In this context, we prove 
two statements of slightly different flavors:
Theorem~\ref{thm:triBK2} is a ``global rigidity statement''.  It says that
if $\v$ is the measurements of a generic configuration $\p$ of $n$ points 
by a set of edges $G$ allowing for trilateration, then there is no 
other set of edges $H$   (maybe not allowing for trilateration)  and
$n$-point configuration $\q$ (maybe not generic) such that measuring $\q$
by $H$ produces $\v$.  

Of course, in an algorithmic setting, we may have no way to know 
in advance either 
the number of points $n$ in the configuration that was measured
or whether the graph $G$ describing the combinatorics of the measurements 
allowed for trilateration.
Our second ``certificate'' statement, Theorem~\ref{thm:triBK2ALG}, says 
the following: 
Let $\p$ be an (unknown) generic configuration and $G$ an
(unknown) edge set producing the (known) measurements $\v$.
Let  
$H$ be
 a (reconstructed) edge set that allows for trilateration and let
$\q$ be a (reconstructed) configuration.
Suppose that the measurements $\v^-$ resulting from $\q$ through $H$
agrees with a subset of $\v$.
Then, in fact, $\q$ is congruent to a subconfiguration of $\p$.
We interpret Theorem~\ref{thm:triBK2ALG} as saying 
that a trilateration-based algorithm can correctly reconstruct (part of) $\p$, without any assumptions 
beyond $\p$ being generic. The algorithm does not even 
need to figure out how to explain all of $\v$, but rather
only needs to find a trilaterizable subset.  For edge measurements,
our unlabeled trilateration algorithm coincides with the 
TRIBOND algorithm described in 
proposed in \cite{dux2}.  However, the analysis of its correctness
is new, and the statements we formulate reveal some subtleties 
that have not appeared before.

These two theorems, while complementary, are incomparable,
because of the different types of assumptions.  Since 
each has a natural application, we will prove analogues 
of both for path and loop ensembles below.

\subsection{Path measurements included}
\label{sec:included}
Suppose next that we want to look at data sets that may include path lengths in addition to edge lengths.
In this case can trilateration still work?  This is a much harder problem.
In particular, suppose we find a subset of $D$ measurements that are consistent with 
the edges of a $K_{d+2}$ subgraph. 
It is conceivable that
there is some adversarial, oddball set of $D$ paths among the $n$ points that, for all configurations, 
can be misinterpreted as
a consistent collection of $D$ measurements from the edges
of a $K_{d+2}$.
Theorem~\ref{thm:bk-linear-s} and Proposition~\ref{prop:simplex2}
are no longer sufficient, as now we are not guaranteed that we are just looking
at a subset of edge measurements. 

To answer these questions, we need a better understanding of 
the behaviour of more general 
linear maps applied to edge lengths over $K_n$. 
(Recall, a path length is just some sum of edge lengths.) 
We do this by introducing a new variety, $L_{d,n}$,
called the unsquared measurement variety.
We study its group of linear automorphisms (so we can apply Theorem~\ref{thm:prin2})
and study which linear maps have images with 
deficient dimension images (so we can apply Theorem~\ref{thm:prin1}).
We have relegated this technical study of these linear automorphism groups to its own
dedicated paper~\cite{loopsAlg}. 
With this in hand, we then argue in this paper that trilateration can still work!

\subsection{Loops instead of paths}

Finally, we now can move to the case where all of our measurements are loops, 
and there are no 
simple edge measurements at all. In reality, most of the hard work though has already been done
by the reasoning of Section~\ref{sec:included}.

To reconstruct a configuration of $d+2$ points, we can no longer use the $D$
edges of a $K_{d+2}$. We instead assume that we have a specific canonical collection
of $D$ loop measurements that we can use to reconstruct $d+2$ points. In fact, we will use
two types of canonical loop sets:  one to reconstruct an isolated $K_{d+2}$, and another
to reconstruct a single new point during trilateration. When describing these, it helps to have dedicated terms for two particular types of loops. So we use \emph{ping} for a loop that contains only two points (and so has length equal to twice an edge length) and \emph{triangle} for a loop that contains three points  (and so has length equal to the sum of three edge lengths).   

To reconstruct an isolated $K_{d+2}$, our hope is to find 
the measurements of lengths comprising $d+1$ pings with one common point,
and the $\binom{d+1}{2}$ triangles that include the pinged $d+1$ points (see Figure~\ref{fig:pt}, bottom left). 
 To add an additional point $v$ during
trilateration, our hope is to use 
all of the edge lengths amongst $d+1$ previously reconstructed points, and to find the measurements of lengths comprising 
one ping and $d$ triangles that include $v$
(see Figure~\ref{fig:pt}, bottom right). 
Thus, in the loop setting, we change our definition of allowing for trilateration to mean that
our loop data includes sufficient canonical data of this type to inductively include all of the points.
In the application described in the introduction, all of the pings and triangles will contain the
common point $\p_1$, but we will not require that assumption in what follows.

With these altered definitions, we can again apply 
the reasoning of Section~\ref{sec:included}
and argue that loop-based trilateration will work as well.

\begin{figure}[ht]
	\centering
	\def\svgwidth{.8\textwidth}
	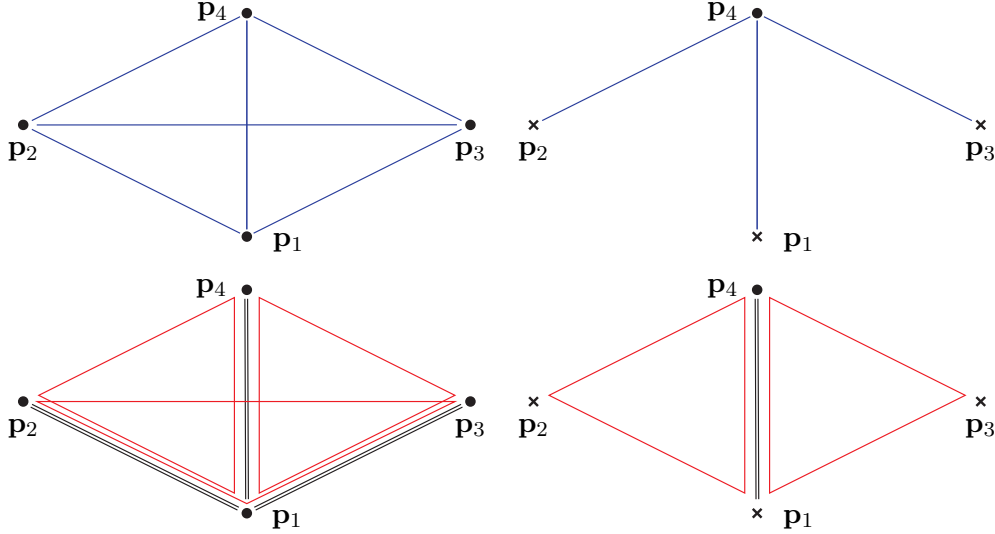
	\caption{Comparison between path measurements (top row) and loop measurements (bottom row). Top row: A $K_{4}$ contained within a path measurement ensemble 
consists of six edges (blue lines) (left). During trilateration using path measurement
 data, three points $\p_1$, $\p_2$, $\p_3$ are known, and a fourth point $\p_4$ is
reconstructed from three edge length measurements (right).
Bottom row: A $K_{4}$ contained within a loop measurement ensemble consists of three pings (double black lines) and three triangles (red lines) (left). During trilateration using loop measurement data, three points $\p_1$, $\p_2$, $\p_3$ are known, and a fourth point $\p_4$ is
reconstructed from one ping and two triangle length measurements (right).}\label{fig:pt}
\end{figure}

%%%%%%%%%%%%%%%%%%%%%%%%%%%
%\newpage

%\tableofcontents

%\newpage
%%%%%%%%%%%%%%%%%%%%%%%%%%%

\section{Definitions and Main Results}\label{sec:defs}

We start by establishing our basic terminology.
\begin{definition}\label{def:constants}
Fix positive integers  $d$  (dimension)
and $n$ (number of points).
Throughout the paper, we will set $N:= \binom{n}{2}$, 
$C:=\binom{d+1}{2}$, and $D:= \binom{d+2}{2}$.
These constants appear often because they are, respectively, the 
number of pairwise distances between $n$ points, the dimension of the 
group of congruences in $\RR^d$, 
and the number of edges in a complete $K_{d+2}$ graph.
\end{definition}

\begin{definition}
\label{def:graph}
A \defn{configuration}, $\pn$ is a  sequence of $n$ points
in $\RR^d$. (If we want to talk about sequences of points in $\CC^d$, we will
explicitly call this a \defn{complex configuration}.) 

We think of each integer in $\{1,\dots,n\}$ as a \defn{vertex}
of an abstract complete graph $K_n$. An \defn{edge}, $\{i,j\}$, is an
unordered distinct pair of vertices. The complete edge set of $K_n$ has cardinality $N$.

A \defn{path} $\alpha:=(i_1,
  \dots, i_z)$ is a finite sequence of 
  $z \geq 2$ vertices, 
  with no vertex immediately repeated.  
  (The simplest kind of path, $(i,j)$, is comprised by a single edge.) 
  A 
\defn{loop} is a path with $z \geq 3$ vertices
where $i_1=i_z$.  
(The
simplest kind of loop $(i,j,i)$ is called a \defn{ping}. Another
important kind of loop $(i,j,k,i)$ is a \defn{triangle}.)

Fixing a configuration $\p$ in $\R^d$, we define the 
\defn{length} of an edge
$\{i,j\}$ to be the Euclidean distance between the points
$\p_i$ and $\p_j$, a real number. 

We define the \defn{length} $v$ of a path or loop $\alpha$ to be
the sum of the lengths of its comprising edges.  
\end{definition}

\begin{definition}
An \defn{edge multiset} is simply
a multiset of edges.
A path or loop naturally gives rise to the edge multiset which contains 
each of the edges traversed, with repetition. 
We will also denote 
an edge multiset as
$\alpha$. 
The length $v$ of an edge multiset $\alpha$ is the sum of the lengths of its comprising edges. We 
denote this measurement process
as $v = \la \alpha, \p \ra$. 
\end{definition}

The motivation for the notation $ \la \alpha, \p \ra$ will become evident in Definition~\ref{def:functionals},
where we describe the measurement process in terms of a linear functional.

\begin{definition}
\label{def:ens}
An \defn{edge measurement ensemble} 
$\pmb{\alpha}:=(\alpha_1,\dots,\alpha_k)$ is a
finite sequence of edges.  (This is the same as an ``ordered graph'' on the vertex set $\{1,\dots,n\}$.)
A \defn{path measurement ensemble} 
$\pmb{\alpha}:=(\alpha_1,\dots,\alpha_k)$ is a
finite sequence of paths. 
We define a 
\defn{loop measurement ensemble} similarly. 
We also define an \defn{edge multiset measurement ensemble} in the same way.

A configuration $\p$ and a measurement
ensemble $\pmb{\alpha}$ give rise to a \defn{data set} $\v$ that is the 
finite sequence 
of real numbers made up of the lengths of its paths or loops or edge multisets. We denote this as 
$\v = \la \pmb{\alpha}, \p \ra$. We say that this data set \defn{arises} from this measurement
ensemble. Notably, a data set $\v$ 
itself does not include any labeling information about the measurement ensemble it arose from.

We denote by $|\v|$ the number of elements in $\v$.
\end{definition}

\begin{definition}
We say that a path or loop $\alpha$ is 
\defn{$b$-bounded}, for some positive integer $b$, 
if no edge appears more than $b$ times in $\alpha$.
We say that a path or loop
measurement ensemble $\pmb{\alpha}$ is $b$-bounded if it comprises only $b$-bounded
loops or paths.
\end{definition}

\begin{remark}
\label{rem:diam}
In a practical setting, we may not know the actual bound $b$ of a $b$-bounded ensemble,
but instead know that it must exist for other reasons. 
In particular, suppose we have some 
bound on the maximal distance between any pair of points
in $\p$. Then we can safely assume that any sufficiently huge
length value $v$ arises from a sufficiently complicated path or loop
that we will not use in our trilateration,
and discard it.
Suppose then that we also have some bound on the minimal distance
between any pair of points in $\p$. Then we know that any non-discarded
value must arise from a $b$-bounded loop or path with some 
appropriate $b$.
\end{remark}

The process of trilateration will involve gluing together smaller pieces of the
configuration. Thus we introduce the following notation.
\begin{definition}
We use $\p_I$ to refer to a \defn{subconfiguration} of a configuration
$\p$ indexed by an index sequence $I$, that is a 
(possibly reordered)
subsequence of
$\{1,\dots,n\}$. In particular, we use $\p_T$ to refer to a
$(d+2)$-point subconfiguration in $\p$, indexed by
a sequence $T=(i_1,\dots,i_{d+2})$ of $\{1,\dots,n\}$. Similarly, we use
$\p_R$ to refer to a $(d+1)$-point subconfiguration  of  $\p$.

We use $\v_J$ to refer to a \defn{sub data set}, 
a (possibly reordered) 
subsequence of the data set $\v$ indexed by an
index sequence $J$, and similarly for a \defn{subensemble} $\pmb{\alpha}_J$.
\end{definition}

We will be interested in measurement ensembles that are sufficient to 
uniquely determine the configuration in a greedy manner using trilateration.
Trilateration starts by finding enough data to reconstruct the location
of ${d+2}$ points. In the path setting, this is done by looking
for the edges of a $K_{d+2}$ graph. In the loop setting, this is done
by looking for a different canonical data set
over $d+2$ points.

\begin{definition}
\label{def:contained}
In the edge or path setting, we say that  
a $K_{d+2}$
 subgraph  of $K_n$ is
\defn{contained} within a path measurement ensemble $\pmb{\alpha}$ if
the ensemble includes a subensemble of size $D$ comprising the edges of this subgraph.
For the $2$-dimensional case, see Figure~\ref{fig:pt} (upper left).

In the loop setting, we say that  
a $K_{d+2}$
 subgraph  of $K_n$ 
with vertices $\{i_1,\dots, i_{d+2}\}$ 
is
\defn{contained} within a loop measurement ensemble $\pmb{\alpha}$ if
the ensemble includes a subensemble of size $D$ comprising:
the $d+1$ pings
$(i_1, i_j , i_1)$ for $j$ spanning  $\{2,\dots,d+2\}$; and also the
triangles 
$(i_1, i_j, i_k, i_1)$
for $j < k$ spanning  $\{2,\dots,d+2\}$.
That is, the ensemble includes
all pings and triangles in this $K_{d+2}$
 with endpoints at vertex
$i_1$.  For the $2$-dimensional case, see Figure~\ref{fig:pt} (bottom left).

\end{definition}

Trilateration proceeds by iteratively adding one more vertex onto an already
reconstructed subset of point locations. This is done by looking for a canonical, sufficient 
set of data. Such data sets differ between the path and the loop setting.

\begin{definition}\label{def:ensemble}
We say that an edge
or a path measurement ensemble \defn{allows for trilateration}%
\footnote{In this definition, and in the rest of the paper, trilaterations 
have a single ``base'' $K_{d+2}$.  One could consider a more general notion that 
allows for multiple bases.  With small modifications, our results carry 
over to that setting.}
if, after reordering the vertices: 
(i) it contains an initial \defn{base} $K_{d+2}$ 
over  $\{1,\dots,d+2\}$; and
(ii) for all subsequent $(d+2) <j \leq n$, it includes as a 
subsequence a \defn{trilateration sequence} comprising the edges
$\{i_1,j\},\dots,\{i_{d+1},j\}$
where all $i_k < j$. For the $2$-dimensional case, see Figure~\ref{fig:pt} (top right).

We say that a loop measurement ensemble \defn{allows for 
trilateration} if, after reordering the vertices:
(i) it contains an initial \defn{base} $K_{d+2}$ over $\{1,\dots, d+2\}$; and
(ii) for all subsequent $(d+2) <j \leq n$, it includes as a subsequence a \defn{trilateration sequence} comprising the triangles  
$(i_1,i_2,j,i_1),\dots,(i_1,i_{d+1},j,i_1)$, and also the ping
$(i_1,j,i_1)$,
where all $i_k < j$.
That is, it includes one ping from $j$ back to one previous $i_1$,
and $d$ triangles back to the previous vertices and including $i_1$.
(See Figure~\ref{fig:pt} (bottom right) for the $2$-dimensional case.)
\end{definition}

\begin{definition}
Trilateration refers to the greedy reconstruction
process whereby one starts by reconstructing a  
base configuration and then inductively adds 
new vertices to the already reconstructed configuration.
We call each such step \defn{trilaterizing a vertex}.

As the process runs, it maintains a set of 
already reconstructed points, and, at each step,
reconstructs a single new point. To find the new point, 
the algorithm searches for an ordered $d+1$-tuple of 
reconstructed points and $d+1$ ordered measurement values 
in the data set so that: the $d+1$ measurements, 
along with the $C = D - (d+1)$ imputed edge lengths among the reconstructed 
points  satisfy a certain
polynomial predicate.
The predicate is based on a Cayley-Menger determinant and proves that these $D$ values arose from $d+2$ points in dimensions $d$.
The $C$ imputed  lengths may or may not be in the data set.

Since $d+1$ of the points are already reconstructed, it is then possible to 
solve uniquely for the last one.
\end{definition}

Note that a path (resp. loop) measurement ensemble that 
allows for trilateration may include any other additional paths (resp. loops) beyond 
those specified in Definition~\ref{def:ensemble}. There 
may also be more than one trilateration sequence
in the ensemble. 

Unlabeled reconstruction from paths or loops
can have difficulties distinguishing between two points with some length between them,
and two points that are half as far from each other but where the edge between them is measured 
twice (as in a ping). Thus we introduce the following scaling notation.

\begin{definition}
For $s$ a real number, the \defn{$s$-scaled} configuration $s\cdot\p$
is the configuration obtained by scaling each of the coordinates of
each point in $\p$ by $s$.  The 
subconfiguration $s\cdot\p_I$
is defined similarly.
For $s$ a positive integer and an
edge multiset $\alpha$, the
\defn{$s$-scaled} edge multiset $s\cdot\alpha$ is 
the edge multiset obtained by scaling the multiplicity of each edge
by $s$.
For $s$ a positive integer, the \defn{$s$-scaled} edge multiset
measurement ensemble $s\cdot\pmb{\alpha}$ is defined by scaling each
element of the ensemble. 

We can then define, for a positive integer $s$, 
the $s$-scaling of a path or loop measurement ensemble by 
considering the elements as edge multisets.
\end{definition}

The main results in this paper do not hold unconditionally. There can be 
special inputs that will fool us or are even inherently ambiguous. 
We explicitly rule out such special inputs in what 
follows.

\begin{definition}
\label{def:genR}
We say that a real point in $\RR^{dn}$
is \defn{generic}
if its coordinates do not satisfy any non-trivial 
polynomial equation with coefficients in $\QQ$.
The set of generic real points have full measure and are 
standard-topology dense in $\RR^{dn}$.

We say that a configuration $\p$ of $n$ points in $\RR^d$ is generic if
it is generic when thought of as a single point in $\R^{dn}$.
\end{definition}

Various theorems in this paper will be shown to hold, 
for each $n$, 
for all generic configurations of the
configuration space, $\RR^{dn}$.  
For example, 
Boutin and Kemper~\cite{BK1,BK2} 
study the question of when an $n$-point configuration in $\R^d$ with $n \geq d+2$
will be uniquely determined from  the complete
set of all $N:=\binom{n}{2}$
edge lengths as an unlabeled set. Their
results show that the configuration will be determined 
unless the coordinates of the configuration satisfy a  polynomial equation 
with rational coefficients
(see Remark~\ref{rem:bkg}).
This means that
such a non-determined configuration must be non-generic.  Contrapositively, genericity
rules out such non-determined configurations.

We note that the set of 
generic points is not open in $\RR^{dn}$.
However, if one imposes any notion of finiteness on the combinatorial objects in question,
such as 
only considering 
measurement ensembles that
are $b$-bounded, for some chosen $b$, then the same 
statements will, in fact, hold over some open and 
dense subset of $\RR^{dn}$. (This set will be Zariski 
open; i.e., the complement of a proper algebraic 
set.)

\begin{remark}
Without some kind of genericity hypothesis, theorems like the ones 
presented in this paper are false: there do exist rare ``bad'' 
inputs for which uniqueness will fail.  Whether it is safe to 
make a genericity assumption depends on the application.  In settings 
where the inputs are unconstrained, such as the sensing ones described 
in the introduction, genericity is a reasonable assumption.
Other applications may impose additional symmetries on the input; for 
example, all the distances might be  drawn from a small set, or the 
structures in question may be invariant to some Euclidean group.
For those applications, whether 
we typically observe the generic behaviour
becomes 
an experimental question. 
\end{remark}

\subsection{Results}\label{sec:results}

The first central conclusion of this paper will be the following
``global rigidity'' 
%(in the sense of  Definition~\ref{def:global-rigidity} below) 
statement:
\begin{mdframed}
\begin{theorem}
\label{thm:punchline}
Let the dimension be  $d\ge 2$. 
Let $\p$ be a generic configuration of $n\ge d+2$ points. Let
$\v= \la \pmb{\alpha}, \p \ra$ where 
$\pmb{\alpha}$
is a path (resp. loop) measurement ensemble
that allows for 
trilateration.

Suppose there is a 
configuration $\q$, 
also of $n$ points,
along with 
an edge multiset
measurement ensemble $\pmb{\beta}$ 
such that
$\v=\la \pmb{\beta},\q \ra$.

Then 
there is a vertex relabeling  of $\q$ such that,
up to congruence,
$s\cdot \q=\p$,
with $s$ an integer $\ge 1$.
Moreover, under this vertex relabeling,
$\pmb{\beta} = s\cdot \pmb{\alpha}$.

\end{theorem}
\end{mdframed}

When 
$\la \pmb{\alpha},\p \ra$ agrees with
$\la \pmb{\beta},\q \ra$ after some permutation, then
the theorem can be applied after appropriately permuting
$\pmb{\beta}$.

Note that if one lets $\q$ be non-generic \emph{and} puts no restrictions on 
the number of points, 
then one can obtain any target $\v$ by 
letting $\pmb{\beta}$ be a tree of edges and then placing
$\q$ appropriately. 

For algorithmic purposes we are better served by the
following variant of Theorem~\ref{thm:punchline}.

\begin{definition}
\label{def:rational:rank}
Given a finite sequence of $k$ complex numbers $w_i$, we say that 
they are \defn{rationally linearly dependent} if there is a
sequence  of rational 
coefficients $c^i$, not all zero, such that 
$0= \sum_i c^i w_i$. Otherwise we say that they are rationally linearly 
independent.
We define the \defn{rational rank} of $w_i$ to
be the size of the maximal subset that is rationally linearly 
independent. 
\end{definition}

\begin{definition}
Suppose that $\q$ is a configuration and 
$\pmb{\beta}$ an
ensemble that allows for trilateration.
We say that a vertex trilateration step is 
\defn{nice} if the $D$ trilaterating length
values ($d+1$ of these are in the data and $C$ 
are imputed from previously  reconstructed vertices) 
used as inputs to the the polynomial predicate
have rational rank $D$.
We say that $\pmb{\beta}$ trilaterates $\q$ \defn{nicely} if 
it has a trilateration sequence such that each step is nice.
\end{definition}

\begin{mdframed}
\begin{theorem}
\label{thm:punchlineALG}
Let the dimension be  $d\ge 2$. 
Let $\p$ be a generic configuration of $n\ge d+2$ points. 
Let
$\v= \la \pmb{\alpha}, \p \ra$
where 
$\pmb{\alpha}$
is an edge multiset measurement ensemble.

Suppose there is a 
configuration $\q$, 
of $n'$ points,
along with 
a path or loop
measurement ensemble $\pmb{\beta}$ 
and let
$\v^-=\la \pmb{\beta},\q \ra$.
Assume that $\pmb{\beta}$ allows for trilateration and
that $\pmb{\beta}$ trilaterates $\q$ nicely.
We also suppose that no two points in $\q$
are coincident.

Suppose that $\v^-$ is contained in $\v$ as a subsequence of values.
Let $\pmb{\alpha}^-$ be the subsequence of 
edge multisets in 
$\pmb{\alpha}$ corresponding to $\v^-$. Let $\p_S$ be the subconfiguration
of $\p$ indexed by the vertices 
that are endpoints of edges in the support  of 
$\pmb{\alpha}^-$.

Then there is a vertex relabeling of $\p_{S}$ such that,
up to congruence,
$s\cdot \p_{S}= \q$,
with $s$ an integer $\ge 1$.
Moreover, under this vertex relabeling,
$s\cdot\pmb{\beta} = \pmb{\alpha}^-$.
\end{theorem}
\end{mdframed}

In this theorem, we make no prior assumptions on 
$\pmb{\alpha}$
and the number of vertices that are endpoints of edges in its support. The existence of $\q$ and 
$\pmb{\beta}$
with the appropriate properties is itself a certificate of correctness
(we still need to assume that $\p$ is generic).
Thus if we are able, in any way, to find a way to 
interpret a portion, which is called $\v^-$ in the statement 
of the theorem, of $\v$ using a nice
trilaterating ensemble $\pmb{\beta}$
we know that we have correctly realized a corresponding
part $\p_S$ of $\p$ (up to similarity).

Theorem~\ref{thm:punchlineALG}
provides the basis for a 
computational attack on this reconstruction process. In particular, we will establish the following.
\begin{mdframed}
Let the dimension be $d\ge 2$. 
Let $\p$ be a generic configuration of $n$ points, and
let
$\pmb{\alpha}$ be a $b$-bounded
path (resp. loop) measurement ensemble
that allows for 
trilateration. 
Suppose $\v = \la \pmb{\alpha}, \p \ra$. Then, given $\v$, there
is a 
trilateration-based algorithm, over a real computation model,
that reconstructs $\p$ up to congruence and vertex labeling.
\end{mdframed}

For fixed $d$,  this algorithm (over a real computation model)
will have worst case time 
complexity that  is polynomial 
in $(|\v|, b)$, though with a moderately large exponent.

\begin{figure}[ht]
	\centering
	\def\svgwidth{.9\textwidth}
	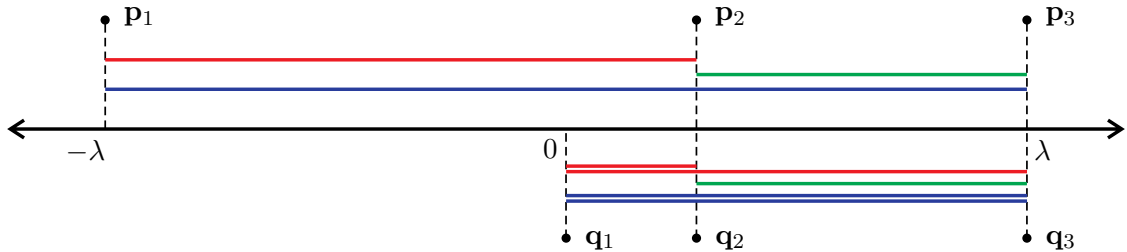
	\caption{Counterexample in $1$ dimension.
  The configuration $\p$ with the shown (upper) three edge measurements
  gives rise to the same length values as the configuration
  $\q$ with the shown (lower) three path measurements. 
  This behavior is stable; as $\p$ is perturbed, $\q$ can
  be appropriately perturbed to maintain this ambiguity, 
  and vice-versa for $\q$ perturbations.    
    }\label{fig:1dambiguity}
\end{figure}

Theorem~\ref{thm:punchline} 
fails for $d=1$. A simple counterexample
to the  theorem for the path case is shown in Figure~\ref{fig:1dambiguity}:
Let $\p_1 < \p_2 < \p_3$ be three generic points on the line.
Let $\alpha_1$ measure the edge $\{1,2\}$, 
$\alpha_2$ measure the edge $\{2,3\}$ and 
$\alpha_3$ measure the edge $\{1,3\}$.
This ensemble clearly allows for trilateration.
In this case we will have 
$\v = \la \pmb{\alpha},\p \ra =
 [ \p_2-\p_1, \p_3-\p_2, \p_3-\p_1]$. 
Now let 
$\q_1$ be arbitrary, and set 
$\q_2 := \q_1 + (\p_2-\p_1) - 1/2 (\p_3-\p_1)$ and 
$\q_3 := \q_1 + 1/2(\p_3-\p_1)$. 
This will give us
$\q_3-\q_2=  (\p_3-\p_1) - (\p_2-\p_1)=
\p_3-\p_2$. Let us also assume that $\p_3-\p_2<\p_2-\p_1$, then this will give us the ordering: $\q_1 < \q_2 < \q_3$. 
Now, let $\beta_1$ measure the path $(2,1,3)$, 
$\beta_2$ measure the edge $\{2,3\}$, and
$\beta_3$ measure the path $(1,3,1)$. 
Then in this case, we will also get 
$\v=\la \pmb{\beta},\q \ra$. But the two 
(edge multiset)
measurement ensembles $\pmb{\alpha}$ and $\pmb{\beta}$
are not related by a scale. 

Likewise, regarding Theorem~\ref{thm:punchlineALG},
let $\q$ be the underlying unknown generic configuration 
measured with the unknown $\beta_i$. 
Looking at these measurement values,
we might incorrectly assume that it comes from the 
reconstructed triangular,
and thus trilaterizable, measurements
described by the $\alpha_i$ on the reconstructed configuration $\p$.
Since we cannot uniquely
reconstruct a triangle on the line, this will kill off 
any attempts at using trilateration for reconstruction in $1$ dimension.

In the language we develop later, 
the failure described in this example essentially happens because  
the variety $L_{1,3}$ is reducible, and
thus Theorem~\ref{thm:prin2} does not apply. The relationship between
these
$\pmb{\alpha}$ and 
$\pmb{\beta}$ 
is not described by a linear  automorphism of $L_{1,3}$. Instead, the relationship is described by only a linear automorphism of one of its (planar) components. 

\begin{figure}[ht]
	\centering
	\def\svgwidth{.9\textwidth}
	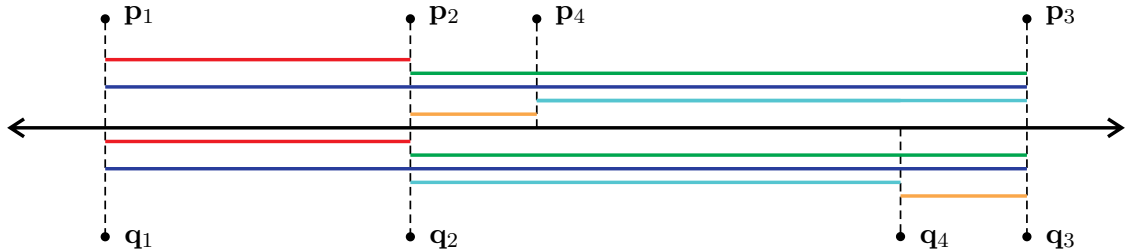
	\caption{$2$-flips ambiguity in $1$ dimension.    
    }\label{fig:2flips}
\end{figure}

There is a second way in which our theorems
fail in $1$ dimension, as demonstrated in 
Figure~\ref{fig:2flips}: 
Let $\p$ consist of $4$ points on a line and
$\pmb{\alpha}$ 
consist of $5$ of the $6$ possible edges. In this case, there
is a vertex, say $\p_4$, with only two measured edges, say
$\{2,4\}$ and $\{3,4\}$. 
If $\pmb{\beta}$
 is obtained from 
$\pmb{\alpha}$ by simply swapping the order of these two edges,
we can maintain $\v$ by appropriately re-locating 
the fourth point.  Essentially, in the unlabeled setting,  there are two ways we can
glue $\p_4$ and its two edges on the triangle of the first
three points. We return to this issue in Remark~\ref{rem:fail}.

We close out this section by remarking that our proofs establish
slightly stronger statements than Theorem \ref{thm:punchline}
and Theorem \ref{thm:punchlineALG}.  In particular, 
the edge multiset measurement ensembles 
$\pmb{\beta}$ in Theorem \ref{thm:punchline} 
and $\pmb{\alpha}$ in Theorem \ref{thm:punchlineALG} 
can, in fact, be any length functional measurement 
ensembles (see Definintion \ref{def:functionals} below), which 
are a generalization of edge multiset measurement ensembles.

\section{Measurement Varieties}\label{sec:varieties}

In this section, we will study the basic properties
of two related families of varieties, the squared and unsquared
measurement varieties. The structure of these varieties
will be critical to understanding 
the problem of 
reconstruction from
unlabeled measurements.
The squared variety is very well studied
in the literature, where it is often called the Cayley-Menger variety,
but the unsquared variety is much less so.
Since we are interested in integer sums of unsquared 
edge lengths, we will need to understand the structure of this
unsquared variety.
Although we are ultimately interested in measuring real lengths
in Euclidean space, we will pass to the complex setting where we can
utilize some tools from algebraic geometry.

\begin{definition}
\label{def:sms}
Let us index the coordinates of $\CC^N$ as $ij$, with
$i < j$ and both between $1$ and $n$.  We also fix an ordering 
on the $ij$ pairs to index the coordinates of $\CC^N$ 
using a single coordinate index with values
between $1$ and $N$.\footnote{This ordering choice does not matter as long 
as we are consistent.  It is there to let us switch between coordinates 
indexed by edges of $K_n$ and indexed using flat vector notation.  For $n=4$, $N=6$ we will 
use the order: $12,13,23,14,24,34$.}
\end{definition}

Let us begin with a 
\defn{complex configuration} $\p$ of $n$ points in $\CC^d$
with $d \geq 1$. 
%We will always assume  $n \geq d+2$.
There are $\edgecard$ vertex pairs (edges), along which we can measure
the complex \emph{squared} length as 
\ba
m_{ij}(\p) := \sum_{k=1}^{d}(\p^k_i-\p^k_j)^2
\ea
where $k$ indexes over the $d$ dimension-coordinates. Here, we measure
complex squared length using the complex squaring operation with no
conjugation. We consider the vector $[m_{ij}(\p)]$ over all of the vertex 
pairs, with $i<j$,
as a single point in $\CC^{\edgecard}$, which we denote as $m(\p)$.

\begin{definition}
Let $M_{d,n}\subseteq \CC^{\edgecard}$ be 
 of $m(\cdot)$ over all $n$-point complex configurations in $\CC^d$. 
This is called the \defn{Cayley-Menger} variety of $n$ points in $d$ dimensions.
We also call this the \defn{squared measurement variety} of $n$ points
in $d$ dimensions.
\end{definition}
When $n \le (d+1)$, then $M_{d,n}= \CC^{\edgecard}$.
(See also Proposition~\ref{prop:infind}
below.)

The next definition, though not needed in what follows, is given for context.
\begin{definition}
If we restrict the domain to  real configurations, then 
we call  under $m(\cdot)$ the \defn{Euclidean squared measurement set} denoted as 
$M^{\EE}_{d,n} \subseteq \RR^{\edgecard}$.
This set has real dimension 
$dn-C$. 
%(In the rigidity literature, this set is often denoted as $M$.)
\end{definition}

The following theorem, save for the last statement,
reviews some basic facts.
For more details, see~\cite{ciprian}
or~\cite{loopsAlg}. 
See Appendix~\ref{sec:geometry}
for our definitions of terms from
algebraic geometry.

\begin{theorem}
\label{thm:Mvariety}
Let $n \ge d+2$.
The set $M_{d,n}$
is linearly isomorphic to 
$\CS^{n-1}_d$,
the variety of complex, symmetric 
%$n-1$-by-$n-1$  
$(n-1)\times(n-1)$
matrices of rank $d$ or less.
Thus, $M_{d,n}$ is a
variety, and also defined over $\QQ$.
It is irreducible.
Its dimension is $dn-C$.
Its singular set $\sing(M_{d,n})$ consists of squared measurements of configurations
with affine spans of dimension strictly less than $d$.
If $\p$ is a generic complex configuration in $\CC^d$ or a 
generic configuration
in $\RR^d$, 
then $m(\p)$ is generic in 
$M_{d,n}$.
\end{theorem}
The last statement follows from Lemma~\ref{lem:genPush}.

\begin{remark}
We note, but will not need, the following:
For $d\ge 1$,
the smallest complex variety containing   
$M^{\EE}_{d,n}$ is $M_{d,n}$.
\end{remark}

As an example, the smallest interesting instance, 
$M_{1,3} \subseteq \CC^3$, is defined by the vanishing of the \defn{Cayley-Menger determinant}, that is, the determinant of the following matrix
\ba
\begin{pmatrix}
2m_{13}& (m_{13} + m_{23} - m_{12})\\
(m_{13} + m_{23} - m_{12})& 2m_{23}\\
\end{pmatrix}
\ea
where we use $(m_{12}, m_{13}, m_{23})$ to represent the coordinates of $\CC^3$~\cite{loopsAlg}.  

In the real setting, the square root of the Cayley-Menger determinant 
of $d+2$ points in $\RR^N$ 
is $(d+1)!2^{d+1}$ times the unsigned $(d+1)$-volume of the 
simplex formed by the points.  In this paper, we are 
interested in the case where $N=d$, so this volume must 
be zero.  Hence, the vanishing of the Cayley-Menger determinant
is a predicate showing that the input measurements arose from 
a configuration in dimension $d$.

Next we move on to unsquared lengths.

\begin{definition}
We define 
the \defn{squaring map} $s(\cdot)$ as the map from $\CC^{\edgecard}$ 
onto $\CC^{\edgecard}$ that
acts by squaring each of the $\edgecard$ coordinates of a point.
Let $L_{d,n}$ be the preimage of $M_{d,n}$ under the squaring map.
(Each point in $M_{d,n}$ has $2^{\edgecard}$ preimages in $L_{d,n}$, arising
through coordinate negations.)
We call this the \defn{unsquared measurement variety} of $n$ points
in $d$ dimensions.
\end{definition}

\begin{definition}
We can define the \defn{Euclidean length map}
of a real configuration $\p$ as
\ba
l_{ij}(\p) := \sqrt{\sum_{k=1}^{d}(\p^k_i-\p^k_j)^2}
\ea
where we use the positive square root.
We denote by $l(\p)$
the vector $[l_{ij}(\p)]$ over all vertex 
pairs.
\end{definition}
The next definition, though not needed in what follows, is given for context.
\begin{definition}
We call  of $\p$ under
$l$ the \defn{Euclidean unsquared measurement set} denoted as 
$L^{\EE}_{d,n} \subseteq \RR^{\edgecard}$.
Under the squaring map, we get
$M^{\EE}_{d,n}$. 
 We may consider $l(\p)$ either as a point in 
the real valued $L^{\EE}_{d,n}$
or as a point in 
the complex variety $L_{d,n}$.
\end{definition}

Indeed, $L^{\EE}_{d,n}$
is the set we are truly interested in,
but it will be easier to work with the whole variety $L_{d,n}$. For example,  Theorem~\ref{thm:prin2}
requires
us to work with varieties, and not, say,
with real ``semi-algebraic sets''.

\begin{remark}
The locus of $\LL$ where the edge lengths of a triangle,
$(l_{12}, l_{13}, l_{23})$, are held fixed is studied in 
beautiful detail in~\cite{marco}, where it is shown to be a Kummer surface.
\end{remark}

The following theorem,
save for the last statement,
is 
a result from our companion paper~\cite{loopsAlg}.
\begin{theorem}
\label{thm:Lvariety}
Let $n \ge d+2$.
$\LN$ is a  variety, defined over $\QQ$. 
It is pure dimensional, with dimension
$dn-C$.
Now additionally assume that $d \geq 2$.
$\LN$ is irreducible.
If $\bm$ is generic in $M_{d,n}$, then each point in 
$s^{-1}(\bm)$ is generic in $\LN$. If $\p$ is a generic configuration in 
$\RR^d$, then $l(\p)$ is generic in $\LN$.
\end{theorem}
The last statement of this theorem follows
from Lemma~\ref{lem:genPull}
and Theorem~\ref{thm:Mvariety}.

In $1$ dimension, the variety $L_{1,3}$ is reducible
and thus  has no generic points. We elaborate on this below.

\begin{remark}
We note, but will not need the following:
For $d\geq 2$,
the smallest complex variety containing   
$L^{\EE}_{d,n}$ is $L_{d,n}$.
\end{remark}

Returning to our minimal example:
The variety $L_{1,3} \subseteq \CC^3$ is defined by the vanishing of the determinant of the
following matrix
\ba
\begin{pmatrix}
2l^2_{13}& (l^2_{13} + l^2_{23} - l^2_{12})
\\
(l^2_{13} + l^2_{23} - l^2_{12})& 2l^2_{23}
\end{pmatrix}
\ea
where we use $(l_{12}, l_{13}, l_{23})$ to represent the coordinates of $\CC^3$.

\begin{figure}[ht]
	\begin{center}
		\includegraphics[width=2.5in]{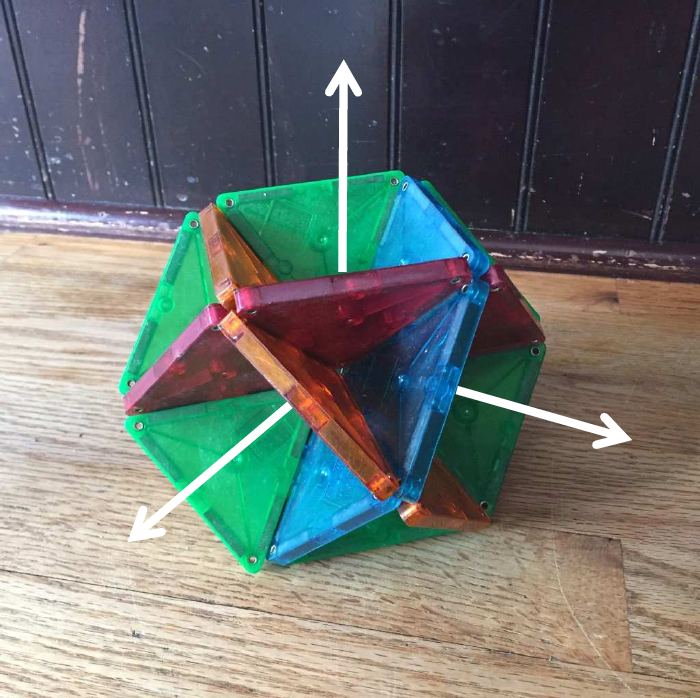}
	\end{center}
	\caption{A model of the real locus of $L_{1,3}$, a subset of $\RR^3$. It comprises $4$ planes. Coordinate axes are in white.}\label{fig:l13}
\end{figure}

\begin{remark}
It turns out that $L_{1,3}$ is reducible and consists of the four hyperspaces defined,
respectively, by the vanishing of one of the following equations:
\ba
l_{12} + l_{23} - l_{13} \\
l_{12} - l_{23} + l_{13} \\
-l_{12} +l_{23} + l_{13} \\
l_{12} + l_{23} + l_{13} 
\ea
This reducibility can make the $1$-dimensional case quite different from dimensions 2 and 3, as already discussed in Section~\ref{sec:results}. See also Figure~\ref{fig:l13}.

Notice that the first octant of the real locus of $3$ of these hyperspaces 
arises as the Euclidean lengths of a triangle in $\RR^1$ (that is,
these make up $L^{\EE}_{1,3}$). 
The specific hyperplane
is determined by the order of the $3$ points on the line.
\end{remark}

\section{The case of edge measurement ensembles}
\label{sec:warm}

We will first look at
the case when our measurement ensemble consists only of
edge measurements. 
Within this, we will start with the case where the measurement ensemble
consists of the complete edge set of cardinality $\edgecard$, which was 
 studied carefully in~\cite{BK1}.
The results from \cite{BK1} have been cleverly applied in~\cite{echo} to determine
the shape of a room from acoustic echo data
(this connection is made explicit in~\cite{dokThesis}).
Then we will consider the case of a trilateration 
ensemble of edges, and prove the correctness of the
TRIBOND
algorithm described in \cite{dux2}.

\subsection{First case: complete graphs $K_n$}

In this section, we will consider 
an {edge measurement ensemble} 
$G:=(E_1,\dots,E_k)$, a
finite sequence of distinct edges of $K_n$.
This is the same thing as a graph on $n$ vertices
with some ordering on its edges.
For an edge measurement ensemble 
$G$, we will write
$\la G, \p \ra^2$ to denote the sequence
of squared edge lengths.

We start with  a central result of
Boutin and Kemper~\cite{BK1}, stated 
in our terminology.

\begin{theorem}
\label{thm:bkMain}
Let the dimension be  $d$, and $n \geq d+2$. 
Let $\p$ be a generic configuration of $n$ points
in $d$ dimensions. Let
$\v= \la G, \p \ra^2$, where 
$G$
is an edge measurement ensemble made up of exactly the $\edgecard$ edges of
$K_n$ in some order.

Suppose there is a 
configuration $\q$, 
also of $n$ points,
along with 
an edge
measurement ensemble $H$, 
where 
$H$
is an edge measurement ensemble made up of exactly the $\edgecard$ edges of
$K_n$, in some other order such that
$\v=\la H,\q \ra^2$.

Then 
there is a 
%unique 
vertex relabeling  of $\q$ such that,
up to 
%unique 
congruence,
$\q=\p$.
Moreover, under this vertex relabeling,
$G=H$.
\end{theorem}
By way of comparison, the labeled setting
is classical.  Recall that  configurations 
$\p$ and $\q$ of $n$ points in dimension $d$ 
are congruent if there is a Euclidean isometry $T$ of 
$\RR^d$ so that $\q_i = T(\p_i)$ for all $1\le i\le n$.
\begin{lemma}[{\cite{YH38}}]\label{lem:mds}
Suppose that $\p$ and $\q$ are configurations
of $n$ points and that for all $N$ edges $ij$ of $K_n$,
we have $|\p_i - \p_j| = |\q_i - \q_j|$.  Then 
$\p$ and $\q$ are congruent.
\end{lemma}

Boutin and Kemper prove
Theorem~\ref{thm:bkMain}  using
a characterization of  
permutation automorphisms of $M_{d,n}$.

\begin{definition}\label{def:linear-automorphism}
A \defn{linear automorphism} of a variety 
$V$ in $\CC^N$ is map that bijectively takes $V$ to itself 
that arises as the
restriction of 
a non-singular linear transformation acting on $\CC^N$.
\end{definition}
In our setting, $V$ will always have a full linear span in $\CC^N$. 
So a linear automorphism on $V$ will uniquely correspond to a linear map 
acting on $\CC^N$. Thus, we may identify a
linear automorphism with a linear map on $\CC^N$.
In our setting, the embedding space $\CC^N$ is equipped
with a fixed set of coordinate axes
that are associated with the edges of $K_n$. 
Thus, we may identify a linear transformation with the complex $N\times N$ matrix 
representing it in the standard basis of $\CC^N$.
We will freely use the symbol $\A$ to represent
a linear automorphism, a linear transformation acting  
on $\CC^N$, or its representing matrix as needed.

\begin{definition}\label{def:gen-perm}
An $N\times N$ matrix $\P$ is a
\defn{permutation matrix} if each row and 
column has a single non-zero entry, and this entry
is $1$.
\end{definition}
\begin{definition}
A permutation $\pi$ of the coordinate axes of $\CC^N$
is \defn{induced by a vertex relabeling} if, under the 
association between the edges of the complete graph $K_n$
and the coordinate axes of $\CC^N$ from Definition \ref{def:sms}, 
there is a permutation $\sigma$ of the vertices of $K_n$ so that, 
for all edges $ij$, $\pi(ij) = \sigma(i)\sigma(j)$.
\end{definition}

\begin{definition}
An $N\times N$ 
permutation matrix $\P$
is 
\defn{induced by a vertex relabeling}
if it  corresponds to a 
permutation, $\pi$, of the edges of $K_n$, that is induced 
by a vertex relabeling.
\end{definition}

The key result of~\cite{BK1} is the following:
\begin{theorem}[{\cite[Lemma 2.4]{BK1}}]\label{thm:bk-linear-s}
Suppose that $\A$ is a permutation matrix
that gives rise to a linear automorphism of $M_{d,n}$.  Then 
$\A$ is induced by a vertex relabeling.
\end{theorem}

Indeed, Theorem~\ref{thm:bk-linear-s} together with Theorem~\ref{thm:prin2}
directly provide a proof for Theorem~\ref{thm:bkMain}.

\begin{remark}
\label{rem:bkg}
Suppose for some specific $(G,\p)$, there is also
an
$(H,\q)$ with the same edge measurements $\v$, 
where $G$ and $H$ are comprised of different orderings of the edges of $K_n$, but
$H$ is not related to $G$ via a vertex relabeling.
Then $\v$ lies in both $M_{d,n}$ and in 
$\P_{HG}(M_{d,n})$, where $\P_{HG}$ is a permutation
matrix that does not give rise to a linear  automorphism of $M_{d,n}$.
So this $\v$, by virtue of it also being in 
$\P_{HG}(M_{d,n})$, is  
not generic in $M_{d,n}$ and thus
$\p$ is not a generic configuration
(Lemma~\ref{lem:genPush}).
\end{remark}

\begin{remark}
\label{rem:noPsdTest}
Theorems~\ref{thm:bk-linear-s} and \ref{thm:prin2}
tell us that, for a generic $\p$,
there will be only one ordering (up to vertex labeling) of the $N$ squared lengths that will 
give us a point in $M_{d,n}$. Testing for membership in $M_{d,n}$
is just a rank test, and does not require any PSD testing. Since, 
by assumption, $\p$ exists and 
is real, its correctly ordered squared lengths must
then automatically satisfy any relevant PSD conditions.
\end{remark}
\subsection{Small images}

We next wish to show that if our point configuration is generic, and we have an ordered subsequence of 
$D$ edge lengths that are consistent with $K_{d+2}$, then indeed, it must arise due to
exactly this subgraph. Using Theorem~\ref{thm:bkMain}, we 
can then uniquely reconstruct these
$d+2$ points.

To this end, 
our first step is to establish  that if the $D$ edges do not
form a $K_{d+2}$ graph, then as we vary over $\p$, we should be able
to vary each of these $D$ numbers independently.

\begin{definition}
\label{def:maps}
Let $d$ be some fixed dimension. 
Let $E:= (E_1,\ldots, E_k)$ be  an edge measurement 
ensemble over $n$ vertices.
The ordering on the edges of $E$
fixes  an association between each edge in $E$ 
and a coordinate axis of $\CC^{k}$. 
Let $m_E(\p) :=
\la E, \p \ra^2$
be the map from $d$-dimensional 
configuration space to $\CC^{k}$
measuring the squared lengths of the edges
of $E$. 

We denote by $\pi_{{E}}$
the linear map  from $\CC^N$
to $\CC^{k}$ 
that forgets the edges not in $E$, and is consistent with the 
ordering of $E$.
Specifically, we have an association between each edge of
$K_n$ and an index in $\{1,\dots,N\}$, and thus we can think of each 
$E_i$ as simply its index in $\{1,\dots,N\}$. Then,
$\pi_{{E}}$ is defined by the conditions:
$\pi_{{E}}(e_j) = 0$ when $j\in \bar{E}$
and
$\pi_{{E}}(e_j) = e'_i$ when $E_i=j$,
where 
$\{e_1, \ldots, e_N\}$ denotes  the coordinate basis for $\CC^N$
and
$\{e'_1, \ldots, e'_k\}$ denotes  the coordinate basis for $\CC^k$.
We call $\pi_{{E}}$ an \defn{edge  map}.

The map $m_E(\cdot)$ 
is simply the composition of the complex measurement map $m(\cdot)$ 
and $\pi_{{E}}$.

Finally, we denote by $M_{d,E}$ the Zariski closure of 
of $m_E(\cdot)$ over all $d$-dimensional configurations. 
\end{definition}

\begin{definition}
\label{def:ind}
We say an edge set  $E$  is \defn{infinitesimally independent} in $d$
dimensions if, starting from a generic complex configuration 
$\p$ in $\CC^d$, we can
\emph{differentially}
vary each of the $|E|$ 
squared lengths independently by
appropriately differentially varying our  configuration $\p$. 
Formally, this means that the 
image 
of the differential of $m_E(\cdot)$ at 
a generic $\p$ is $|E|$-dimensional.
This exactly
coincides with the notion of infinitesimal independence
from graph 
rigidity theory~\cite{L70}.

An edge set that is 
not infinitesimally independent in $d$ dimensions is 
called \defn{infinitesimally dependent} in $d$ dimensions.
Note that in this case the rank of the differential, $dm_E$,  can never rise to 
$|E|$.
\end{definition}

The following is implicit in the rigidity theory literature. 
\begin{proposition}
\label{prop:infind}
An edge measurement ensemble
$E$ is infinitesimally  independent in $d$ dimensions
iff  of $m_E(\cdot)$ over all complex 
configurations of $n$ points has dimension $|E|$.  
\end{proposition}
\begin{proof}[Proof sketch]
The basic principle is that the generic rank of the differential tells us the dimension of .
In particular we consider the map $m_E(\cdot)$.
First remove the non-smooth points of , and then remove
the preimages
of these non-smooth points from the domain (all non-generic).
Sard's Theorem (e.g.,~\cite[Theorem 14.4]{harris}) then 
tells us that the inverse image of 
every generic point in this image
consists entirely
of configurations $\p$ where the differential
has rank equal to the dimension of  of $m_E(\cdot)$.
\end{proof}

\begin{remark}\label{rem:complex-rigid}
These notions are usually studied in the real setting, but the
tools used in the proof sketch above 
work the same way in the complex setting.  Complexification 
is used to study rigidity problems in, e.g., \cite{ST10,cgr,OP07,ciprian}.
\end{remark}

The following is a standard result from rigidity theory.  
\begin{proposition}
\label{prop:simplex}
Let $E$ be an edge measurement ensemble (with all its edges
distinct).
Suppose $|E| \le \binom{d+2}{2}$ and $E$ is 
infinitesimally dependent
in $d$ dimensions.
Then $|E| = \binom{d+2}{2}$ and
$E$ consists of the edges of a  $K_{d+2}$ subgraph (in some order).
\end{proposition}
\begin{proof}[Proof Sketch]
Assume, without loss of generality, that $E$ is infinitesimally
dependent and inclusion-wise minimal with this property.  
If $E$ does not consist of the edges of a $K_{d+2}$ subgraph, 
then it has a vertex $v$ of degree at most $d$. Suppose that 
$\p$ is a configuration in affine general position.
A geometric argument then shows that 
the coordinate subspaces 
spanned by the edges incident to $v$ must be 
in  of $dm_E$ at $\p$.  Since  of 
$dm_E$ at $\p$ is not $|E|$-dimensional, removing 
the edges incident to $v$ yields a smaller set of 
edges $E'\subseteq E$ that is still infinitesimally dependent.
This  contradicts the assumed minimality of $E$.
\end{proof}

Combining Propositions~\ref{prop:infind} and~\ref{prop:simplex} we arrive at the following:
\begin{proposition}
\label{prop:simplex2}
Let $E$ be an edge measurement ensemble (with all its edges
distinct).
Suppose $|E| \le \binom{d+2}{2}$ and 
 of 
$m_E(\cdot)$ over all complex 
configurations of $n$ points has dimension less than $|E|$.  
Then $|E| = \binom{d+2}{2}$ and
$E$ consists of the edges of a  $K_{d+2}$ subgraph (in some order).
\end{proposition}

\subsection{Consistent with  $K_{d+2}$}

We can now complete the argument that when our data looks consistent
with a single $K_{d+2}$, then we can be certain that this must be
how this data arose. The key idea is that unless we were measuring 
 a $K_{d+2}$, then for a generic $\p$, and using Proposition~\ref{prop:simplex2},
 we should not expect to find a measurement that satisfies any extra algebraic
 condition defined using rational coefficients.

\begin{proposition}
\label{prop:ind}
Let the dimension be  $d\ge 1$. 
Let $\p$ be an 
$n$-point configuration in $\RR^d$
such that $m(\p)$ is generic in $M_{d,n}$.
Suppose there is a sequence of $D$, not-necessarily distinct,  edges $E = (E_1, \ldots, E_D)$,
such that $w_i:=\la E_i,\p \ra^2$ form a vector
$\w:=(w_1, \dots, w_D)$
that
is in $M_{d,d+2}$.
%with no two of the $w_i$ identical. 

Then
there must be a  subconfiguration  
$\p_T$ of $\p$ with $d+2$ points 
such that
$\w = m(\p_T)$.
%the squared edge measurement 
%action of the $E_i$ on $\p$ computes
%$\w = m(\p_T)$.
\end{proposition}
The condition that $m(\p)$ is generic in $M_{d,n}$ allows
for the possibility that $\p$ itself is non-generic. 
Here, we are essentially only
requiring that $\p$ is congruent to a generic configuration.
Since the proof requires several lemmas, we defer it for now.

\begin{remark}
\label{rem:consistPSD}
This proposition only requires that $\w \in M_{d,d+2}$ which 
is a zero-determinant test; no PSD test is required.
The conclusion tells us that these squared lengths come from a real
configuration, and so must automatically satisfy the relevant
PSD conditions.
\end{remark}

In the statement, we do not assume, a priori, that the $D$ edges in $E$ are distinct.
So we will prove first that under our genericity
assumption on $\p$, this will be automatically 
guaranteed.

\begin{lemma}\label{lem:1}
Let $d\ge 1$.  Let $I = \{I_1, \ldots, I_{r}\}$, 
for $r <  D$,
be a partition of $\{1, \ldots, D\}$ into  
$r$ subsets.  Let $X_I\subseteq \CC^D$ be the linear span of the 
vectors $\v_j = \sum_{i\in I_j} \e_i$, for $1\le j\le r$ and 
elementary vectors $\e_i$; i.e., these 
are vectors with only $r$ distinct values.

Then $M_{d,d+2}$ does not contain the space 
$X_I$.
\end{lemma}
\begin{proof}
Observe that, for any partition $I$ of the coordinates, the 
all ones vector is a linear combination of the vectors $\v_j$.  
Hence any $X_I$ contains the all ones vector.

On the other hand, it is well-known that the $(d+2)$-simplex with 
all edge lengths equal to one has non-zero $(d+1)$-volume.  Hence 
the all ones vector is not in $M_{d,d+2}$.
\end{proof}

\begin{lemma}\label{lem:2}
Let the dimension be $d\ge 1$. 
Let $\p$ be a $n$-point configuration in $\RR^d$
such that $m(\p)$ is generic in $M_{d,n}$.
Suppose there is a sequence of $D$, not-necessarily distinct,  edges $E = (E_1, \ldots, E_D)$,
such that $w_i:=\la E_i,\p \ra^2$ form a vector
$\w:=(w_1, \dots, w_D)$
that
is in $M_{d,d+2}$.

Then the $D$ edges must be distinct.
\end{lemma}
\begin{proof}
Let $\E$ be the $D$ by $N$ matrix with $ij$-th entry
set to $1$ if $E_i$ is the $j$th edge of $K_n$, and $0$
otherwise. 
We identify $\E$ with the linear map it induces from $\CC^N$ to $\CC^D$.
%We can then restrict its action on $M_{d,n}$.
Then $\E(m(\p))$ gives us 
the $D$ measurements $w_i=\la E_i,\p \ra^2$.

From Theorem~\ref{thm:prin1}, since $m(\p)$ is generic in $M_{d,n}$, 
we see that 
$\E(M_{d,n}) \subseteq M_{d,d+2}$.

Suppose that the $D$ edges comprising
$E$ are not all distinct.  
Then $E'$, the collection of $r<D$
distinct edges, is infinitesimally independent (Proposition~\ref{prop:simplex}).
From Proposition~\ref{prop:infind}, 
 $\E(M_{d,n})$
must be an $r$-dimensional constructible subset  $S$ of
one of the $r$-dimensional  linear (and irreducible)
spaces $X_I$ described in Lemma \ref{lem:1}.
Its Zariski closure, $\bar{S}$, would then be equal to  this
$X_I$.  Since $M_{d,d+2}$ contains $S$ and is Zariski 
closed, $M_{d,d+2}$ must contain $X_I$.
% This closure, $\bar{S}$, must be contained
% in any variety, such as $M_{d,d+2}$, that contains $S$.

But from Lemma~\ref{lem:1}, $X_I$ is not contained 
in $M_{d,d+2}$. This contradiction establishes the lemma.

\end{proof}

\begin{proof}[Proof of Proposition~\ref{prop:ind}]

From Lemma~\ref{lem:2} we know that the edges in $E$ are all distinct.
Our measurement sequence $\w$ arises from  $D$ distinct coordinates of $m(\p)$,
giving us
$\w = 
\pi_{{E}}(m(\p))$.  Since $\p$ is generic, $m(\p)$ is generic in 
$M_{d,n}$.
Recall also that, from Theorem~\ref{thm:Mvariety},
$M_{d,n}$ is irreducible.

From Theorem~\ref{thm:prin1}, 
we see that $\pi_{{E}}(M_{d,n}) \subseteq M_{d,{d+2}}$, since
$\pi_{{E}}(m(\p))\in M_{d,{d+2}}$.

Since  dimension 
of $\pi_{{E}}(M_{d,n})$ is less than 
$D$, 
from Proposition~\ref{prop:simplex2} (which required us to know that the edges
were distinct), we see 
that $E$ must consist of the edges of a $K_{d+2}$.

Let $\pi_{{K}}$ be the edge  map 
where $K$ comprises the edges 
of this $K_{d+2}$, and where $K$ is ordered 
such that 
$\pi_{{K}}(M_{d,n}) = M_{d,d+2}$. 
Let $\K$ be the matrix representing $\pi_{K}$,
and let  $\E$ be the matrix representing $\pi_{E}$.

Then we must have 
%$\pi_{{E}}= \P \pi_{{K}}$
$\E =
\P \K$
where $\P$ is some $D\times D$  permutation matrix.
We identify $\P$ with the linear map that it induces on $\CC^D$.
And we have 
$\P(M_{d,d+2}) = 
\P(\pi_{{K}}(M_{d,n})) = 
\pi_{{E}}(M_{d,n}) 
\subseteq  M_{d,d+2}$.

Thus, from Theorem~\ref{thm:prin2}, $\P$
must induce a linear  automorphism on $M_{d,d+2}$.
Then from Theorem~\ref{thm:bk-linear-s}, 
$\P$ must be induced from a vertex relabeling. 
As  
%$\pi_{{E}}=\P \pi_{{K}}$ 
$\E=\P \K$ 
for such 
a $\P$, 
there must be an ordered $(d+2)$-point
subconfiguration $\p_T$ of $\p$ such that 
$\w = m(\p_T)$.
\end{proof}

\begin{remark}
Because of the way Proposition~\ref{prop:simplex2} is 
used in the above proof, $K_{d+2}$ cannot simply be
replaced by some other generically globally rigid 
graph to obtain a similar result.
\end{remark}

\subsection{Trilateration}
\label{sec:tri1}

Now we wish to extend this result to the case where our edge 
measurement ensemble 
is not complete, but does allow for trilateration in the sense of
Definition~\ref{def:ensemble}. 
Due to the gluing ambiguity we saw in Figure~\ref{fig:2flips}, 
we will restrict our discussion here to $d\ge 2$.
The key idea is to use Proposition~\ref{prop:ind}
iteratively, applied to $K_{d+2}$
subsets of $K_n$. This will lead to the following
global rigidity theorem. 
\begin{mdframed}
\begin{theorem}
\label{thm:triBK2}
Let the dimension be  $d\ge 2$. 
Let $\p$ be a generic configuration of $n\ge d+2$ points. Let
$\v= \la G, \p \ra^2$ where 
$G$
is an edge measurement ensemble
that allows for 
trilateration. 

Suppose there is a 
configuration $\q$, 
also of $n$ points,
along with 
an edge
measurement ensemble $H$,
such that 
$\v=\la H,\q \ra^2$.

Then 
there is a vertex relabeling  of $\q$ such that,
up to congruence,
$\q=\p$.
Moreover, under this vertex relabeling,
$G=H$.
\end{theorem}
\end{mdframed}

We note that this result has recently been
greatly strengthened to apply
to a much larger class of graphs than just
the complete graphs or  trilateration graphs~\cite{gugr}.
The proof we give here, for trilateration graphs, is 
much more direct, as it is based on greedily constructing the relabeling 
of $\q$, one point at a time.

The following variant   describes a certificate for correct
reconstruction.

\begin{mdframed}
\begin{theorem}
\label{thm:triBK2ALG}
Let the dimension be  $d\ge 2$. 
Let $\p$ be a generic configuration of $n \ge d+2$ points. Let
$\v= \la G, \p \ra^2$ where 
$G$
is an edge measurement ensemble.

Suppose there is a 
configuration $\q$, 
of $n'$ points,
along with 
an edge
measurement ensemble $H$ that allows for trilateration
such that 
$\v^-=\la H,\q \ra^2$.
We also suppose that no two points in $\q$
are coincident.

Suppose that $\v^-$ is contained in $\v$ as a subsequence of values.
Let $G^-$ be the subsequence of edges in $G$ corresponding to
$\v^-$. Let $\p_S$ be the subconfiguration
of $\p$ indexed by the vertices
that are endpoints of edges in the support of $G^-$.

Then 
there is a vertex relabeling  of $\p_S$ such that,
up to congruence,
$\q=\p_S$.
Moreover, under this vertex relabeling,
$G^-=H$.
\end{theorem}
\end{mdframed}

The assumption that no two points in $\q$ are coincident is
required. Otherwise one could create a $\q$ that is identical
to $\p$ except that one high valence vertex in $(G,\p)$ 
is split
in $(H,\q)$ into two distinct vertices with coincident locations,
each with enough edges to the rest of $\q$ so that both new vertices can
be trilaterated.

In this theorem, we make no prior assumptions on $G$
and its number of vertices. Nor do we assume, a priori, that $\q$ is generic.
The existence of $\q$ and $H$
with the appropriate properties is itself a certificate of correctness,
though we still need to assume that $\p$ is generic.
Thus if we are able to  
interpret some portion of $\v$, corresponding 
to $\v^-$ in the statement 
of the theorem,  using a 
trilateration $H$, then we know that we have correctly 
realized a corresponding part $\p_S$  of $\p$.

Finally, suppose  we assume 
that $\p$ is a generic configuration of $n$ points
and that $G$ allows for trilateration with
$\v= \la G, \p \ra^2$.
 Thus we must
be able to take the data $\v$, and find (using brute force)
some $H$ and $\q$ of $n$ points
such that $H$ trilaterates  $\q$,
and 
such that 
$\v=\la H,\q \ra^2$. From Theorem~\ref{thm:triBK2ALG}, 
we then know that $\q=\p$. This gives us a formal justification
for TRIBOND,
the brute force unlabeled trilateration algorithm of~\cite{dux2}.

\paragraph{Theorems~\ref{thm:triBK2} and~\ref{thm:triBK2ALG}, sketch}
Informally, Theorems~\ref{thm:triBK2} and~\ref{thm:triBK2ALG} say that, generically, 
we do not need to know the edge labels for the 
trilateration reconstruction to succeed.

The intuitive reasoning is as follows.  The trilateration 
process starts from a known $K_{d+2}$ and then locates
each additional point by ``gluing'' a new $K_{d+2}$ (with one
new point) onto a $K_{d+1}$ inside the already visited $K_{v}$
over the $v$ previously reconstructed vertices.
The idea, then, is to find the labels as we locate points
by using Proposition~\ref{prop:ind}
iteratively: initially 
to find a ``base'' $K_{d+2}$ 
to start the trilateration 
process, and then, after measuring all the edges between the 
visited points, to find subsequent $K_{d+2}$ subconfigurations
that, each, add one more
point.  When $d\ge 2$, there is only one way 
to do the gluing, because generic $(d+1)$-simplices
do not have any ``self-congruences''.

Even though the steps above are conceptually very simple,
the details require some care. 
We now fill in the sketch above.

\begin{lemma}
\label{lem:noSim}
Suppose that $\p$ is a configuration of $n$ points in dimension $d$ so that
either $l(\p)$ is generic in $L_{d,n}$ or $m(\p)$ is generic 
in $M_{d,n}$.  Then no two subconfigurations of 
at least three points in $\p$ are similar to each other, 
unless the two subconfigurations consist of the same points,
 in the same order.
\end{lemma}
 Only the case of congruence is needed now, but similarities will be needed later in Section~\ref{sec:ugr}. 
\begin{proof}
Two ordered configurations $\q$ and $\r$ of $k$ points are related by a similarity 
if and only if 
their vectors of $\binom{k}{2}$ (un)squared edges lengths are 
proportional (in the squared case, 
if the scaling in the similarity is $\lambda$, the effect on 
$m(\q)$ is to multiply it by $\lambda^2$).  
That is, if and only if both of the $\binom{k}{2}\times 2$
matrices 
\[
\begin{pmatrix}
    m(\q) & m(\r)
\end{pmatrix}
\qquad \text{and}\qquad 
\begin{pmatrix}
    l(\q) & l(\r)
\end{pmatrix}
\]
have 
rank at most one, which is a polynomial condition defined over $\QQ$ that is 
non-trivial when $k \ge 3$ (so that $\binom{k}{2} > 1$).

If $\q$ and $\r$ are 
similar 
subconfigurations of $\p$ with $k\ge 3$ points, 
then, by the argument above $l(\p)$ and $m(\p)$ 
satisify a non-trivial polynomial condition 
that does not hold over all of $L_{d,n}$ and 
$M_{d,n}$ respectively (there are configurations where 
$\q$ and $\r$ are not similar).  Hence, $l(\p)$ and $m(\p)$ 
are non-generic when $\p$ contains similar subconfigurations.
\end{proof}
\begin{remark}
\label{rem:fail}
The statement of Lemma~\ref{lem:noSim} is not true 
with only two points because any pair of two-point configurations
are similar.  Even worse for our intended application, 
any subconfiguration $(\p_i, \p_j)$ is congruent to
the subconfiguration $(\p_j, \p_i)$.
Because of this,
unlabeled trilateration over an edge measurement ensemble
will not directly work for $d=1$. (See Figure~\ref{fig:2flips}.)
In order to use trilateration over an unlabeled 
edge measurement ensemble in $1$ dimension, 
we would need to have an edge ensemble that 
not only allows for trilateration in $1$ dimension but 
also has enough edges to 
allow for trilateration in $2$ dimensions.
\end{remark}

The next lemma describes our main inductive step.

\begin{lemma}\label{lem:partial-trilat}
Let $d\ge 2$ and let $\p$ and $\q$ be configurations
of $n$ and $n'$ points respectively, 
with 
%so that $\p$ is generic and 
$m(\p)$ generic in 
$M_{d,n}$.  
We also suppose that no two points in $\q$
are coincident.
Let 
$G$ and $H$ be two edge measurement ensembles such that
$\la G, \p\ra^2 = \la H, \q\ra^2$.

Suppose
that we have two 
``already visited'' subconfigurations
$\q_{V'}$ and $\p_{V}$ with 
$\q_{V'} =\p_{V}$.

Suppose  we can 
find a set $F$  of  
$d+1$ 
%distinct 
edges in 
$H$ 
connecting some unvisited
vertex $\q_{i'}\in \q_{\bar{V'}}$ to 
some visited 
subconfiguration $\q_{R'}$ of $\q_{V'}$ with
$d+1$ vertices. 

Then we can find an unvisited
$\p_{i}\in \p_{\bar{V}}$ such that 
the two subconfigurations $\q_{V'\cup\{i'\}}$
and $\p_{V\cup\{i\}}$ are equal.
\end{lemma}
\begin{proof}
Let $\q_{T'}$ be a subconfiguration consisting of, in some order, all the
points of $\q_{R'}$ along with $\q_{i'}$.
Let $\w:=m(\q_{T'})$.

The $d+1$ points of $\q_{R'}$ give rise to $C$ edge length measurements
in $\w$.  Because $\p_{V} = \q_{V'}$ the corresponding $d+1$ points 
$\p_R$ induce the same measurements.  By assumption, there must 
be $D - C = d+1$ edges $E_0\subseteq G$, corresponding to the the $d+1$
edges $F\subseteq H$ connecting the points of $\q_{R'}$ to 
$\q_{i'}$, so that $\la E_0,\p\ra^2 = \la F,\q\ra^2$.  
Adding the $C$ imputed edges from $\p_V$ to $E_0$ we get a 
set of $D$ edges $E\subseteq G$ 
(not necessarily distinct),
so that $\w = \la E,\p\ra^2$.

Since $m(\p)$ is generic, we can now apply 
Proposition~\ref{prop:ind} using 
the existence of $E$  to 
conclude that there must be a  
subconfiguration $\p_{T}$ of $\p$
with
$\w = m(\p_{T})$.
Now, we use  Lemma~\ref{lem:mds}
to  conclude that
$\q_{T'}$ and $\p_{T}$ 
are related by a 
congruence.  
(Now that we have $\p_{T}$, we don't need $E$
any more.)

Since $\q_{R'}$ is a subconfiguration of $\q_{T'}$, 
$\q_{R'}$ must be congruent to its associated 
subconfiguration 
of $\p_{T}$, which we may call
$\p_{R_0}$.
% From genericity of $m(\p)$ in $M_{d,n}$,
% we have, using Lemma~\ref{lem:genPull}, that 
% $l(\p)$ is generic in $L_{d,n}$.
Thus from Lemma~\ref{lem:noSim},
we know that 
$\q_{R'}$ is congruent to no other
subconfiguration of $\p$. 
Meanwhile, 
$\q_{R'}$ is a subconfiguration of $\q_{V'}$ and thus
also equal to  some subconfiguration 
$\p_{R}$
of $\p_{V}$.
Thus $\q_{R'}$, 
$\p_{R_0}$ and $\p_{R}$ must all be 
equal.
Since the congruence $\sigma$
that maps $\q_{T'}$ to $\p_{T}$
fixes the $d+1$ points of $\q_{R'}$, $\sigma$
must be the identity and we must have
$\q_{T'}=\p_{T}$.

Let $\p_{i}$ be the ``new'' point in $\p_{T}\setminus \p_{R}$, which also
must equal $\q_{i'}$.
If $\p_{i}$ was already visited in $\p_{V}$, 
then the same position would have already been visited
by some point in $\q_{V'}$. This together with the fact that no points are
coincident in $\q$ would contradict the assumption that 
$\q_{i'}\in \q_{\bar{V'}}$.
Thus 
$\q_{V'\cup\{i'\}}=\p_{V\cup\{i\}}$.
\end{proof}

We can now apply the above 
lemma iteratively.

\begin{lemma}
\label{lem:triBK}
Let the dimension $d \ge 2$. 
Let $\p$ be  a
configuration of  $n$ points 
such that $m(\p)$ is generic in $M_{d,n}$.
Let $G$ be a measurement ensemble.
Let $\v :=\la G,\p \ra^2$.

Suppose that $\q$ is a 
configuration of $n'$
points
with no two points in $\q$ being coincident.
And suppose that 
$H$
is an edge measurement ensemble
that allows for 
trilateration
and such that 
$\la H, \q \ra^2$ also equals $\v$.

Then, there is a sequence of indices $S=(s_1,\ldots, s_{n'})$ 
so that, up to congruence, $\p_{S} = \q$.  Moreover, 
the vertices appearing in $S$ are exactly those that 
are 
endpoints of edges in the support of $G$.
After renaming each edge $\{i,j\}$ in $H$ as $\{s_i,s_j\}$, 
we have $H = G$.
\end{lemma}
\begin{proof}
For the base case, 
the trilateration assumed in $H$
guarantees a $K_{d+2}$ contained in $H$, over
a $(d+2)$-point subconfiguration  $\q_{T'}$ of  $\q$.
Define $\w:=m(\q_{T'})$.
We have 
$\q \in M_{d,d+2}$.

Using the
fact that
$\la G, \p\ra^2 = \la H, \q\ra^2$ 
we can apply Proposition~\ref{prop:ind} to this $\w$,  $\p$
and appropriate sequence of edges $E$ taken from $G$.
From this, we conclude 
that
there is a ${d+2}$ point
subconfiguration  
$\p_{T}$ of 
%$\q_{S'}$ 
$\p$ 
such that
%$m_{E'}(\q_{S'})=\w=m(\q_{T'})$.
$\w=m(\p_{T})$.
From Lemma~\ref{lem:mds},
up to congruence, we have
$\q_{T'} = \p_{T}$.

Going forward, 
assume that  this congruence  
has been factored into $\p$. 
Then, to proceed inductively, assume that we have  
two ``visited'' subconfigurations such that 
$\q_{V'}=\p_{V}$.  Initially $V = T$ and $V' = T'$.
With this setup, we may now follow the trilateration of $\q$, 
iteratively applying
Lemma~\ref{lem:partial-trilat} until we have visited all of  
$\q$.  At the end of the process, $\q_{V'} = \p_{V}$ with 
$\q_{V'}$ a reordering of $\q$.
Inverting this ordering, we have $\q=\p_{S}$, where $S$ is 
an ordering of the visited points in $\p$.

Since $\p$ is generic, then no two distinct edges  can have the same squared length. 
Since $\q$
is a subconfiguration of our generic
$\p$
then no two distinct edges among points
in $\q$
can have the same squared length. 
This means there is a \emph{unique} way for  $\v$ to arise from
$\q$ and a unique way for $\v$ to arise from $\p$.  Hence, 
after vertex relabeling from $S$, we have $H = G$.
Since the vertices
that are endpoints of edges in the support of $H$ correspond exactly to the points of $\q$, 
then the vertices
that are endpoints of edges in the support of $G$
correspond to the points of $\p_{S}$,
as in the statement.

\end{proof}

And we can now finish the proof of one of the main theorems of this section:
\begin{proof}[Proof of Theorem~\ref{thm:triBK2ALG}]
First we remove from $\v$ the measurements which do not
appear in $\v^-$. We also remove the associated edges from
$G$, to obtain $G^-$.
Since $\p$ is generic, then $m(\p)$ is 
generic in $M_{d,n}$ from Theorem~\ref{thm:Mvariety}.
Then we simply apply Lemma~\ref{lem:triBK}.
\end{proof}

With some other added assumptions, 
we can use an assumption of genericity 
on $\p$ to
automatically obtain genericity
for $m(\q)$ in $M_{d,n}$. To see this, we first use the following definition. 

\begin{definition}
Let $d$ be a fixed dimension.
Let $E$ be an edge measurement ensemble
 with $n\ge d+1$.
We say $E$ is 
%A ordered set of $E$ of edges graph $\Gamma$ with $n \ge d+1$ vertices and $|E|$ edges is 
\defn{infinitesimally rigid} in $d$ dimensions,
if, starting at some  generic 
(real or complex) configuration $\p$,
there are no differential  motions of $\p$ in $d$ dimensions that preserve 
all of the squared lengths among the edges of $E$,
except for differential congruences.

When an edge measurement ensemble is infinitesimally  rigid, then the lack
of differential motions holds over a Zariski
open subset of
configurations that includes all generic configurations.

Letting $m_E(\p)$ be the map from
configuration space to $\CC^{|E|}$
measuring the squared lengths of the edges
of $E$, 
infinitesimal rigidity means 
that  of the differential of $m_E(\cdot)$ at $\p$ is $(dn-C)$-dimensional.

\end{definition}
The following proposition follows exactly as Proposition~\ref{prop:infind}.
\begin{proposition}
\label{prop:infrig}
If $E$ is  infinitesimally rigid, then  of $m_E(\cdot)$ acting on all configurations
is $(dn-C)$-dimensional. Otherwise, the dimension of  is smaller.
\end{proposition}

\begin{lemma}
\label{lem:triBK2}
In dimension $d \ge 1$, let  $\p$ and $\q$ be two 
configurations with the same number of points $n \ge d+1$.
Suppose that 
$G$ and
$H$
are two edge measurement ensembles, each with 
the same number $k$ of edges,
and with $G$
infinitesimally  rigid in $d$ dimensions.
And suppose that  
$\v:=\la G,\p \ra^2 =
\la H,\q \ra^2$.

If $\p$ is a generic configuration, then 
$m(\q)$ is generic in $M_{d,n}$.
\end{lemma}
\begin{proof}
Recall the notation introduced in Definition~\ref{def:maps}.
The varieties  $M_{d,G}$ and  $M_{d,H}$, both subsets of 
$\CC^k$,
are 
defined over $\QQ$. 
They are irreducible
since they arise from closing images of 
$M_{d,n}$,
which is irreducible, under 
polynomial (in fact linear) maps.
Because $G$ is infinitesimally rigid,
$M_{d,G}$
is of 
dimension $dn-C$ from Proposition~\ref{prop:infrig}.
Likewise, $M_{d,H}$ is of dimension at most
$dn-C$. 
% if 
% $H$ is infinitesimally rigid, otherwise it is 
% of smaller dimension.

Our assumptions give us $\v \in M_{d,G}$
and
$\v \in M_{d,H}$.

We claim $M_{d,G} = M_{d,H}$.
If not, then 
$M_{d,G} \cap M_{d,H}$ is an algebraic
variety, defined over $\QQ$, of dimension strictly less 
than $dn-C$ (due to irreducibility),  
and thus could contain no generic points of $M_{d,G}$.
But we have assumed that $\v$ is in both, and thus also in
this intersection set. But since $\p$ is generic,
then $\v$ is  generic in $M_{d,G}$ (Lemma~\ref{lem:genPush}).
This contradiction thus establishes our claim.

Since
$M_{d,G} = M_{d,H}$, then $\v$ is also a generic point of 
of $M_{d,H}$.

Finally, since $M_{d,H}$ is the 
Zariski closure of the 
image of $M_{d,n}$ under 
the linear map $\pi_{{H}}(\cdot)$,
and since they 
have the same dimension, 
then from Lemma~\ref{lem:genPull} 
the preimage of $\v$
under $\pi_{{H}}(\cdot)$, which is $m(\q)$,
must be a generic point in 
$M_{d,n}$. 
\end{proof}

And we can now finish the proof of the other main theorem of this section:
\begin{proof}[Proof of Theorem~\ref{thm:triBK2}]
An edge measurement ensemble that allows for 
trilateration is always infinitesimally
rigid.  
We apply Lemma~\ref{lem:triBK2} to conclude that $m(\q)$ is generic in $M_{d,n}$. 
By assumption $G$ allows for trilateration.
Thus we can now apply Lemma~\ref{lem:triBK}, but with the roles of $\p$ and $\q$
reversed, as well as the roles of $G$ and $H$.
\end{proof}

\section{Paths and Loops}

Now we are ready to tackle the setting of paths and loops. 
Our reasoning will parallel that of Section~\ref{sec:warm}, but we will
need to upgrade most of the ingredients.

\subsection{First case: complete graphs $K_n$}
Our first step is to 
upgrade Theorem~\ref{thm:bk-linear-s}.
Dealing with the possibility of paths will
make us look at 
$L_{d,n}$ instead of $M_{d,n}$. Moreover this will 
require us to understand
the full group of linear automorphisms, instead of the simpler
setting of coordinate permutations.
This is a non-trivial algebraic geometry study, which
we have relegated to a companion paper~\cite{loopsAlg}. 
Here we summarize those results.

Any coordinate permutation that is 
induced by a vertex relabeling
must give rise to a linear automorphism of $L_{d,n}$.
Also due to the squaring construction, the
negation of any coordinate will give rise to a
linear automorphism. 
We call the linear automorphisms arising  
from vertex relabelings and 
coordinate negations the 
\defn{signed vertex relabelings}.
Any uniform scale on $\CC^N$
will also  give rise to  a linear automorphism of $L_{d,n}$
because uniformly 
scaling a configuration $\p$ in $\RR^d$ by a positive factor
produces another configuration 
with edge lengths scaled by the same amount.
Let us refer to the linear automorphism set generated by the above
as the \defn{expected} linear automorphisms of $L_{d,n}$. 

In our application, we are concerned that the measurement 
data of some complicated set of paths could yield a generic 
point in $L_{d,n}$.
If this were to happen, there 
would necessarily be a linear automorphism of $L_{d,n}$ 
that is \emph{not} one of the expected ones.  The 
following theorem rules this out, except 
when $\{d,n\} = \{2,4\}$.

\begin{theorem}[\cite{loopsAlg}]\label{thm:no-regges}
Let $d \ge 1$ and let $n \ge d+2$.  
Assume that $\{d,n\}\neq\{2,4\}$.
Then any linear automorphism $\A$ of $L_{d,n}$ is a scalar multiple
of a signed vertex relabeling.
\end{theorem}

Interestingly, for $L_{2,4}$, there is, up to scale,  a discrete
set of unexpected linear automorphisms, which we can fully characterize.
Luckily for us, 
when an unexpected linear automorphism is representable  by a matrix with 
real elements, this matrix must include
some negative coefficients.
Such negative numbers cannot occur 
in any measurement ensemble based on paths or loops.

A linear automorphism  $\A$ of $L_{d,n}$ is \defn{real} if its matrix has
only real entries and \defn{non-negative} if its matrix contains only 
real and non-negative entries.

\begin{theorem}[\cite{loopsAlg}]
\label{thm:auto2}

The group of real linear automorphisms of $L_{2,4}$,
up to positive scale,
 is of order $23040$
and is isomorphic to the Weyl group $D_6$.  The subset of 
this group that are represented by matrices consisting of  
non-negative elements 
is a 
subgroup of order $24$ and acts by relabeling the vertices of $K_4$.

\end{theorem}

\subsection{Small images}
\label{sec:maps}

Proposition~\ref{prop:simplex2} above is a theorem about
edge measurement ensembles where  over all configurations is low
dimensional. Here we need a similar proposition that applies to linear maps acting on 
$L_{d,n}$.
This will also require some non-trivial analysis of 
linear maps acting on $L_{d,n}$ which we relegate to the companion paper~\cite{loopsAlg}.

Let $d \ge 1$.
Recall that $D:= \binom{d+2}{2}$.
In this section,
$\E$ will be 
a $D\times \edgecard$ matrix of rank $r$
where 
$r$ is some number at most $D$.
We will also identify the symbol $\E$ with  
the linear map from 
$\CC^N$ 
to
$\CC^D$ that it induces.
We will restrict this map to its action on $\LN$.
Our goal is to study these restricted linear maps where the 
dimension of  is strictly less than $r$. 
In particular, this will occur when $\E(\LN) = L_{d,d+2}$.

\begin{definition}
We say that $\E$ has 
$K_{d+2}$ support 
if it depends only on measurements supported over the $D$ 
edges corresponding to a $K_{d+2}$ subgraph of $K_n$. 
Specifically,  
all the columns of the matrix 
$\E$ are zero, except for at most $D$ of
them, and these non-zero columns index edges 
contained within a single $K_{d+2}$.
\end{definition}

The following  
is proven in~\cite{loopsAlg}.
\begin{theorem}
\label{thm:linImage}
Let $d\ge 1$ and $n\ge d+2$.
Let $\E$ be  a 
$D\times \edgecard$ matrix with rank $r$.
Suppose that  $\E(L_{d,n})$,
a constructible set,  
is not of dimension $r$.
Then  $r=D$ and $\E$ has $K_{d+2}$ support.
\end{theorem}
\begin{remark}
Theorem~\ref{thm:linImage} does not hold when 
$\LN$ is replaced by $M_{d,n}$. 
The linear automorphism group of
$\CS^{n-1}_d$ arises from  all the maps of the form
\[
	\S\mapsto \A^t\S\A 
\]  
with $\A$ an invertible $(n-1)\times (n-1)$ matrix.

Correspondingly, there are linear 
automorphisms of $M_{d,n}$ arising from matrices 
$\A$
that have
dense support. Thus, even if some $\E$ has $K_{d+2}$ support
the matrix $\E\A$ would not, and it 
could still have a small-dimensional image. 

If, in our motivating application, we observed sums of 
\emph{squared} edge lengths, we would have no hope to 
recover $\p$ without some combinatorial information.
\end{remark}

\subsection{Consistent $K_{d+2}$}
\label{sec:consist}

Now we can use these ingredients to upgrade Proposition~\ref{prop:ind}.
We want to show that:  
If we take $D$ values from our data set of path measurements, and they
are \defn{consistent} with the $D$
edge lengths of a $K_{d+2}$ in $\RR^d$, then in
fact \emph{they do arise}, up to scale, in this way.
Likewise, in the loop setting, if they 
``look like'' an appropriate \defn{canonical} set of  $D$
loops, 
then in
fact \emph{they do arise}, up to scale, in this way.

First we generalize our notion of 
measurement ensembles from Definition~\ref{def:ens}.

\begin{definition}\label{def:functionals}
A \defn{length functional} $\alpha$ is a linear mapping from $\LN$ to $\CC$.
We write its application to $\bl \in \LN$ as $\la \alpha,\bl \ra$.
In coordinates, it has the form $\sum_{ij} \alpha^{ij} l_{ij}$, with $\alpha^{ij} \in \CC$.
When $\p$ is a real configuration, and thus $l(\p)$ is well defined,
then we  also define
$\la \alpha, \p \ra:= \la \alpha, l(\p)\ra$.
Similarly, 
we can define a 
\defn{functional measurement ensemble} 
$\pmb{\gamma}$.
Then we define 
$\la \pmb{\gamma}, \bl \ra$ and
$\la \pmb{\gamma}, \p \ra$
to be the sequence that arises from the application of the measurement ensemble
$\pmb{\gamma}$ to $\bl$ and $l(\p)$ respectively.

We can apply to a length functional the adjectives: 
\defn{rational}, \defn{non-negative}, 
\defn{integer} or  \defn{whole} (non-negative integer) if all of its coordinates have these properties.
We say that an integer or whole length functional 
is \defn{$b$-bounded} if all of its coordinates have magnitudes  
no greater than $b$.

An edge multiset  naturally gives rise
to a unique whole length functional. Analogously, an edge multiset measurement ensemble gives
rise to a unique whole length ensemble matrix.

\end{definition}

\begin{theorem}
\label{thm:consist}
Let $d \ge 2$.
Let $\p$ be a   $n$-point configuration in $\RR^d$
such that $l(\p)$ is generic in $\LN$.
Suppose there is a sequence of 
$D$ non-negative rational functionals $\gamma_i$
such that $w_i=\la \gamma_i,\p \ra$ form a 
(necessarily non-negative)
vector
$\w:=(w_1,\dots,w_D)$ that has 
rational rank
$D$  and is in 
$L_{d,d+2}$.

Then
there must be a $(d+2)$-point subconfiguration 
$\p_T$ of $\p$ 
such that
$\la \pmb{\gamma}, \p \ra = \w = s\cdot l(\p_T)$, where $s$ is a positive
scale factor. 
\end{theorem}

\begin{proof}
We can use the $D$ functionals, $\gamma_i$, as the rows of a matrix $\E$. We identify this matrix with the linear map from
$\CC^N$ to $\CC^D$ that it induces. 
%We restrict this map to its action on $\LN$.
This matrix $\E$ maps
$l(\p)$ to $\w$, which we have assumed to be in 
$L_{d,d+2}$.
Since $d\ge 2$, then from Theorem~\ref{thm:Lvariety}, $L_{d,n}$ is irreducible.

 From Theorem~\ref{thm:prin1} 
we see that 
$\E(\LN) \subseteq L_{d,d+2}$. 

From Lemma~\ref{lem:depC} and the assumed rational rank of
$\w$,
the rank of $\E$ must be $D$.
Since  of $\E(\LN)$ is not $D$-dimensional,
then from Theorem~\ref{thm:linImage},
$\E$ is $K_{d+2}$ supported.

Let $\pi_{{K}}$ be the edge  map 
where $K$ comprises the edges 
of this $K_{d+2}$, and where $K$ is ordered 
such that 
$\pi_{{K}}(L_{d,n}) = L_{d,d+2}$.
Let $\K$ be the matrix representing $\pi_{K}$.
Then $\E$ can be written in the form  
$\A \K$ where
$\A$ is a $D \times D$ non-singular and
non-negative rational matrix.

We have   $\A(L_{d,d+2}) = \A(\pi_{{K}}(L_{d,n})) 
= \E(L_{d,n})
\subseteq L_{d,d+2}$.
Thus, from Theorem~\ref{thm:prin2}, $\A$ must be the matrix of  a
real non-negative linear
automorphism of $L_{d,d+2}$. From 
Theorems~\ref{thm:no-regges} (for $d \ge 3$)
and \ref{thm:auto2} (for $d=2$)
$\A$
arises from 
a permutation on $d+2$ vertices and a positive global scale.
As $\E=\A\K$ for such an $\A$, 
there must exist an (ordered) $(d+2)$-point subconfiguration $\p_T$
of $\p$, such that 
$\la \pmb{\gamma}, \p \ra = \w = s\cdot l(\p_T)$.
\end{proof}

Theorem~\ref{thm:verLoop}, in 
Section~\ref{sec:canon} below, is the
generalization to the case of loop ensembles. 
 
\begin{remark}
\label{rem:scale}
If the $\gamma_i$ are 
whole valued, then $s$ must be an integer, 
greater than or equal to $1$
for any such
$\p_T$.
This is because, in the proof above, the matrices $\E$ and 
$\A$ will be integer valued.
\end{remark}

\begin{remark}\label{rem:rat-rank}
The rational rank $D$ hypothesis is essential as the following example in dimension $2$ shows.  Let 
$\gamma_i$ be any functional, and measure using the ensemble
$(3\gamma_i, 4\gamma_i, 
5\gamma_i, 5\gamma_i, 4\gamma_i, 3\gamma_i)$. 
These measurement values (with rational rank $1$) correspond
to a $K_4$ made by gluing ``345 triangles'' together, no matter what $\p$ is.
In fact, using an arithmetic 
construction from \cite{AE45}, we can even make infinitely many 
non-congruent rational rank $1$ measurement sets that have no 
repeated measurement values or $3$ collinear points.

When we are proving  ``global rigidity'' results in Section~\ref{sec:ugr}, 
the assumed trilateration sequence automatically gives rational rank $D$.
On the other hand, a reconstruction algorithm 
(see Section~\ref{sec:recon}) based on Theorem~\ref{thm:consist} 
will have to find the trilateration sequence as it goes.  
The examples here show that such an algorithm has to test rational rank.
\end{remark}

\subsection{Loop Setting}
\label{sec:canon}

We also wish to alter Theorem~\ref{thm:consist} so that 
it can be applied to the loop setting.
In particular, instead of looking for  
measurements of  $D$ edges of a $K_{d+2}$ we will
look for $D$ \emph{canonical}
measurements over a $K_{d+2}$. 
Indeed, we consider two such canonical measurements:
one to identify a $K_{d+2}$ ex-nihilo, and one to identify 
a $K_{d+2}$ using a known $(d+1)$-point subconfiguration along with 
$d+1$ additional measurements.

\begin{definition}
Given a single $K_{d+2}$, with ordered vertices and edges,
we can describe the  $D$
measurements described in Definition~\ref{def:contained}
using a fixed \emph{canonical}
$D\times D$ matrix $\N^d_1$. Each row represents
the edge multiplicities of one measurement loop.
For notational convenience, we order the rows of this matrix so that
each of the first $C$ rows is supported  only over the $C$ edges of the first
$d+1$ points.

In $2$ dimensions, 
we associate the columns of this matrix with the 
following edge
ordering:
$\{1,2\}, \{1,3\}, \{2,3\}, \{1,4\}, \{2,4\}, \{3,4\}$.
This then gives us
the following \defn{tetrahedral measurement matrix} that measures
three pings and three triangles
\ba
\N^2_1:=
\begin{pmatrix}
2&0&0& 0&0&0\\
0&2&0& 0&0&0\\
1&1&1& 0&0&0 \\
0&0&0& 2&0&0 \\
1&0&0& 1&1&0 \\
0&1&0& 1&0&1 
\end{pmatrix}.
\ea
See Figure~\ref{fig:pt} (bottom left) for the $2$-dimensional case.

Given an initial $K_{d+1}$, with ordered vertices and edges,
we can describe an ordered $D$
measurements describing the trilateration of an additional  vertex
off of the first $d+1$ vertices, as defined in Definition~\ref{def:ensemble},
using fixed a $D\times D$ matrix $\N^d_2$. 
The first $C$ rows measure the  edges of the initial $K_{d+1}$
subconfiguration, and the remaining $d+1$ rows 
measure the appropriate pings and triangles. 

In $2$ dimensions, this gives us
the following \defn{trilateration measurement matrix}
that measures three edges,
one ping, and two triangles
\ba
\N^2_2:=
\begin{pmatrix}
1&0&0& 0&0&0\\
0&1&0& 0&0&0\\
0&0&1& 0&0&0 \\
0&0&0& 2&0&0 \\
1&0&0& 1&1&0 \\
0&1&0& 1&0&1 
\end{pmatrix}.
\ea
See Figure~\ref{fig:pt} (bottom right) for the $2$-dimensional case.
\end{definition}

\begin{customthm}{\ref*{thm:consist}$'$}
\label{thm:verLoop}
With these definitions in place,
we can alter the condition
in Theorem~\ref{thm:consist} 
that $\w$ describes a point of $L_{d,d+2}$,  to the condition that
$\w$ describes a of 
point of $\N^d_i(L_{d,d+2})$. And we can generalize the
conclusion to read:
$\la \pmb{\gamma}, \p \ra = \w = s\cdot  \N^d_i(l(\p_T))$.
\end{customthm}
\begin{proof}
We follow the structure of the proof of Theorem~\ref{thm:consist}.
From Lemma \ref{lem:bij}, $\N^d_i(L_{d,d+2})$ is an irreducible 
variety.
Our loop assumptions give us 
$\w = \E(\bl) \in \N^d_i(L_{d,d+2})$ and thus from genericity of $\bl$, 
$\E(L_{d,n})\subseteq \N^d_i(L_{d,d+2})$. 

As above, $\E$ has $K_{d+2}$  support.
Let $\pi_{{K}}$ be the edge  map 
where $K$ comprises the edges 
of this $K_{d+2}$, and where $K$ is ordered 
such that 
$\pi_{{K}}(L_{d,n}) = L_{d,d+2}$.
Let $\K$ be the matrix representing $\pi_{K}$.
The matrix 
$\E$ can be written in the form  
$\B \K$ where
$\B$ is a $D \times D$ non-singular and non-negative rational
matrix.

We have   $\B(L_{d,d+2}) = \B(\pi_{{K}}(L_{d,n})) 
= \E(L_{d,n})
\subseteq \N^d_i(L_{d,d+2})$.
It follows 
from Theorem~\ref{thm:prin2}
that $\A:=(\N^d_i)^{-1}\B$ 
is the matrix of a linear automorphism of 
$L_{d,d+2}$. Thus we are left with determining the 
linear automorphisms with matrices $\A$ such that 
$\N^d_i\A=\B$ is non-negative.

For $d\ge 3$, from Theorem~\ref{thm:no-regges},
we see  
that this only occurs when $\A$ is 
a positive scale of a vertex relabeling.
For $d=2$, we need to explicitly
do a non-negativity check on our $\N^2_i\A$ over all 
linear automorphism matrices $\A$ of $L_{2,4}$,
as characterized in Theorem~\ref{thm:auto2} and described in detail
in~\cite{loopsAlg}.
This only requires checking $23040$ matrices,
for both $\N^2_1$ and $\N^2_2$
and has been done in the Magma computer algebra system~\cite{loopsAlg}. 
In both cases, non-negativity only arises for 
$\A$ that are positive scales of a vertex relabeling. 
%(See supplemental script.)

Thus $\E$ is of the form
$\N^d_i\A\K$ for such an $\A$, and the result follows.
\end{proof}

\subsection{Trilateration}
\label{sec:ugr}

Now we can upgrade our results in Section~\ref{sec:tri1} to prove
our main theorems.

\begin{lemma}\label{lem:partial-trilat2}
Let $d\ge 2$ and let $\p$ and $\q$ be configurations
of $n$ and $n'$ points respectively, 
so that 
%$\p$ is generic and 
$l(\p)$ is generic in 
$L_{d,n}$.  
We also suppose that no two points in $\q$
are coincident.
Let
$\pmb{\alpha}$ 
be an edge multiset measurement ensemble and
$\pmb{\beta}$ be a path or loop ensemble such that
$\la \pmb{\alpha}, \p\ra = \la \pmb{\beta}, \q\ra$.

Suppose
that we have two 
``already visited'' subconfigurations
$\p_V$ and $\q_{V'}$ with 
$\p_V =\q_{V'}$.

Suppose  we can 
find $\pmb{\delta}$, a subset of  
$d+1$ paths or loops in 
$\pmb{\beta}$, that can used to
nicely trilaterate 
some unvisited
vertex $\q_{i'}\in \q_{\bar{V'}}$ over 
some visited ${d+1}$ point
subconfiguration $\q_{R'}$ of $\q_{V'}$.

Then we can find an unvisited
$\p_{i}\in \p_{\bar{V}}$ such that 
the two subconfigurations $\p_{V\cup\{i\}}$
and $\q_{V'\cup\{i'\}}$ are equal.

\end{lemma}
\begin{proof}
Let $\q_{T'}$ be a subconfiguration consisting of, in some order, all
the points of $\q_{R'}$ along with $\q_{i'}$. 
Let $\w := \N(l(\q_{T'}))$,
with $\N=\I$ in the path setting and 
$\N=\N^d_2$ in the loop setting.
We have $\w \in L_{d,d+2}$. By assumption, $\w$  
has  rational rank $D$.

We have
$\la \pmb{\alpha}, \p\ra = \la \pmb{\beta}, \q\ra$ and
$\p_V = \q_{V'}$.  Thus, the assumption that 
$\pmb{\delta} \subseteq \pmb{\beta}$ nicely trilaterates 
an unvisited vertex in $\q$ implies that 
we can use $\pmb{\alpha}$ to 
find a measurement ensemble $\pmb{\gamma}$ 
of $D$ measurements so that $\la \pmb{\gamma}, \p\ra = \w$.
The ensemble $\pmb{\gamma}$ plays the role of $E$ in 
the proof of Lemma \ref{lem:partial-trilat}, and it 
is constructed in a similar way: from $d+1$ measurements
in $\pmb{\alpha}$ and $C$ distances imputed from a set of 
$d+1$ points in $\p_V$.

Theorem~\ref{thm:consist} in the path setting (or 
Theorem~\ref{thm:verLoop} in the loop setting) can be applied using 
$\p$ together with this 
$\pmb{\gamma}$ and $\w$.
This guarantees a 
$(d+2)$-point subconfiguration 
$\p_{T}$ of $\p$ such that
$\w = t \cdot \N(l(\p_{T}))$,
with $\N=\I$ in the path setting and  $\N=\N^d_2$ in the loop 
setting. 
%Here $t \ge 1$ is an integer scale factor (see Remark~\ref{rem:scale}). 
From Lemma~\ref{lem:mds}, 
we  
conclude that $\p_T$ and $\q_{T'}$ are related by a similarity.  

Since $\q_{R'}$ is a subconfiguration of $\q_{T'}$, 
$\q_{R'}$ must be similar to the associated 
subconfiguration 
of $\p_{T}$, which we may call
$\p_{R_0}$.
From genericity of $l(\p)$  in $L_{d,n}$
and Lemma~\ref{lem:noSim}, 
$\q_{R'}$ is similar to no other
subconfiguration of $\p$. 
Meanwhile, 
$\q_{R'}$ is a subconfiguration of $\q_{V'}$ and thus
also equal to  some subconfiguration 
$\p_{R}$
of $\p_{V}$.
Thus $\p_R$, $\p_{R_0}$ and $\q_{R'}$
must all be equal.
Since the similarity $\sigma$
that maps $\q_{T'}$ to $\p_{T}$
fixes the $d+1$ points of $\q_{R'}$, $\sigma$
must be the identity and 
$\p_T=\q_{T'}$.

Let $\p_{i}$ be the ``new'' point in $\p_{T}\setminus \p_{R}$,
which must be equal to $\q_{i'}$.
If $\p_{i}$ was already visited in $\p_{V}$, then the same position would have already been visited by some point in $\q_{V'}$. This together with the fact that no points are coincident in $\q$ would contradict 
the assumption that $\q_{i'}\in \q_{\bar{V'}}$. Thus 
$\p_{V\cup\{i\}}=\q_{V'\cup\{i'\}}$.
\end{proof}

Applying the above iteratively yields the following:

\begin{lemma}
\label{lem:punchline1}
Let  $d \ge 2$. 
Let $\p$ be a configuration of $n$ 
points 
such that $l(\p)$ is generic in $L_{d,n}$.
Let $\pmb{\alpha}$ be 
an edge multiset
measurement ensemble. 
Let $\v:=\la \pmb{\alpha},\p \ra$.

Suppose that there is a 
configuration $\q$ of $n'\ge d+2$ points
with no two points in $\q$ coincident.
And suppose that
$\pmb{\beta}$ is a path (resp. loop) measurement ensemble
that allows for trilateration,  and such that
$\pmb{\beta}$ trilaterates $\q$ nicely, 
and such that 
$\la \pmb{\beta}, \q \ra$ also equals $\v$.

Then, there is a sequence of indices $S$ so that, up to congruence,
$s\cdot \p_{S} = \q$, with $s$ an integer $\ge 1$.
Moreover, 
the vertices appearing in $S$ are exactly those that are 
endpoints of edges in the support of $\pmb{\alpha}$. 
Under the 
vertex relabeling implicit in $S$, we have 
$\pmb{\alpha} = s\cdot \pmb{\beta}$.
\end{lemma}
\begin{proof}

For the base case, 
the nice trilateration assumed in $\pmb{\beta}$
guarantees a $K_{d+2}$ contained in  $\pmb{\beta}$
over a $(d+2)$-point  subconfiguration  $\q_{T'}$ of  $\q$.
Define 
$\w :=   \N(l(\q_{T'}))$
with $\N=\I$ in the path setting and  $\N=\N^d_1$ in the loop 
setting. From the niceness assumption, we can 
find such a $K_{d+1}$ so that these $w_i$ have rational rank $D$. 
We have 
$\w \in \N(L_{d,d+2})$.

Using the fact that
$\la \pmb{\alpha}, \p\ra = \la \pmb{\beta}, \q\ra$ 
we can apply 
Theorem~\ref{thm:consist} in the path setting, or 
Theorem~\ref{thm:verLoop} in the loop setting,
to this $\w$ and $\p$ 
and an appropriate subensemble $\pmb{\gamma}$ of $\pmb{\alpha}$.
We conclude 
that
there is a $(d+2)$-point 
subconfiguration  $\p_{T}$ of 
%$\q_{S'}$ 
$\p$ 
such that
$\w = s \cdot \N(l(\p_{T}))$,
with $\N=\I$ in the path setting and  $\N=\N^d_1$ in the loop 
setting.
Here $s \ge 1$ is an integer scale factor (see Remark~\ref{rem:scale}).
Also,  
from Lemma~\ref{lem:mds},
up to a similarity, we have
$\p_T = \q_{T'}$.

In order to use the simpler terminology of equality instead of 
similarity going forward, we will 
apply the similarity to $\p$
and replace 
each functional  $\alpha$ in $\pmb{\alpha}$ by 
$(1/s)\alpha$.
Then, to proceed inductively, assume that we have 
two ``visited'' subconfigurations such that 
$\p_{V}=\q_{V'}$.  Initially $V = T$ and $V' = T'$.
With this setup, we may now follow the trilateration of $\q$, 
iteratively applying
Lemma~\ref{lem:partial-trilat2} until we have visited all of  
$\q$.  At the end of the process, $\p_V = \q_{V'}$ with 
$\q_{V'}$ a reordering of $\q$.
Inverting this ordering, we have $\q=\p_{S}$, where $S$ is 
an ordering of the visited points in $\p$.

Since $l(\p)$ is generic in $L_{d,n}$, 
from Theorem~\ref{thm:functional}, no two distinct functionals can give
the same measurement. 
Since $\q$ is a subconfiguration of our generic $\p$, then 
no two distinct functionals can give the same measurement.
This means there is a \emph{unique} way for  $\v$ to arise from
$\p$ and a unique way for $\v$ to arise from $\q$.  Hence, 
after vertex relabeling from $S$, we have
$\pmb{\alpha}=
\pmb{\beta}$.
Since the vertices with endpoints in the 
support of $\pmb{\beta}$ correspond exactly to the points of $\q$, 
then the vertices that are endpoints of the edges in the  support of $\pmb{\alpha}$ are exactly $S$, as in the statement.
\end{proof}

Now we can prove one of our main theorems.

\begin{proof}[Proof of Theorem~\ref{thm:punchlineALG}]
First we remove from $\v$ the measurements which do not
appear in $\v^-$. We also remove the associated edge multisets from
$\pmb{\alpha}$.
Since $\p$ is generic, then $l(\p)$
is generic in $L_{d,n}$ from 
Theorem~\ref{thm:Lvariety}.
Then we simply apply 
Lemma~\ref{lem:punchline1}. 
\end{proof}

Next we want to use a genericity assumption on
$\p$ to automatically obtain 
genericity for $l(\q)$ in $L_{d,n}$
under the assumption that $n'=n$.

\begin{definition}
Let $n \ge d+1$.
A path or loop  measurement  ensemble
is \defn{infinitesimally rigid} in $d$ dimensions
if, starting at some (equiv. any) generic 
real 
configuration $\p$,
there are no differential motions of $\p$ that preserve 
all of the measurement values, except for differential congruences.
\end{definition}

\begin{remark}
We defined infinitesimal rigidity of an edge measurement ensemble 
in terms of a generic real of complex configuration $\p$.  For 
paths and loops, since we are using unsquared measurements, we  
restrict our attention to real configurations, so that the measurement 
map can be defined.
\end{remark}

\begin{lemma}
\label{lem:punchline2}
In dimension $d\ge 2$, let  $\p$ and $\q$ be two 
configurations with the same number of points $n \ge d+1$.
Suppose that 
$\pmb{\alpha}$
is a path or loop
measurement ensemble that is 
infinitesimally rigid in $d$ dimensions.
Suppose that 
$\pmb{\beta}$
is an edge multiset
measurement ensemble.
And suppose that  
$\v:=\la \pmb{\alpha},\p \ra =
\la \pmb{\beta},\q \ra$.

If $\p$ is a generic configuration, 
then $l(\q)$ is generic in $L_{d,n}$.
\end{lemma}
\begin{proof}
Since $d\ge2$, we know that  $L_{d,n}$ is irreducible (Theorem~\ref{thm:Lvariety}), and so too is 
any linear image of it 
(see Definition \ref{def:cons}).
We define
$L_{d,\pmb{\alpha}}$ and 
$L_{d,\pmb{\beta}}$ 
to be the Zariski closures of s of $L_{d,n}$ under the 
linear maps corresponding to $\pmb{\alpha}$
and $\pmb{\beta}$, 
respectively. They are both defined over $\QQ$.
We  define the measurement map
$l_{\pmb{\alpha}}(\p)$
acting
on (real) configuration space
to be the map that takes
$\p$ to $\la {\pmb{\alpha}}, l(\p)\ra$.

If $\pmb{\alpha}$ is  infinitesimally rigid then, 
as in the proof of 
Proposition~\ref{prop:infind},
we see that the real dimension of the
 image $s$ of 
$l_{\pmb{\alpha}}(\p)$ over all real 
configurations
is $dn-C$ (constant rank theorem).  
The set $s$ is semi-algebraic (using quantifier elimination).
Let $V$ be the (complex) Zariski closure of $s$.
From Lemma~\ref{lem:zclose}, the complex dimension of
$V$ equals the real dimension of $s$.
Since we have 
\(
    s\subseteq L_{d,\pmb{\alpha}} 
\)
and $L_{d,\pmb{\alpha}}$ is Zariski closed, we have 
\(
V \subseteq L_{d,\pmb{\alpha}} 
\).
Using this containment and the fact that $L_{d,\pmb{\alpha}}$
a linear image of $L_{d,n}$ we obtain
\[
    dn - C \le \dim L_{d,\pmb{\alpha}} \le 
    \dim L_{d,n} = dn - C,
\]
which shows that equality holds throughout.
Using the fact that $L_{d,\pmb{\beta}}$
a linear image of $L_{d,n}$ we obtain
$\dim L_{d,\pmb{\beta}} \le dn - C$.

The proof then follows  like that of Lemma~\ref{lem:triBK2}.

\end{proof}

And now we can prove our other main theorem.

\begin{proof}[Proof of Theorem~\ref{thm:punchline}]
By assumption, $\pmb{\alpha}$ allows for trilateration.
A path or loop 
measurement ensemble that allows for 
trilateration is always infinitesimally
rigid.  (The ability to trilaterate makes  a generic real
$\p$  uniquely determined (globally rigid) from its labelled 
measurements. Thus, one cannot continuously change $\p$ while
keeping the measurements fixed (local rigidity).  
If the ensemble was not infinitesimally rigid
at a generic $\p$, then using the 
proof from~\cite{asimow}, it could not be locally rigid.)
So from Lemma~\ref{lem:punchline2} we know that $l(\q)$ is generic in $L_{d,n}$.
We next argue that this trilateration must be nice:
During a trilateration step, the $D$ functionals
are linearly independent. So from Lemma~\ref{lem:Cdep}
and the genericity of $\p$, the $D$ measurement values
must have rational rank $D$.

So we can now apply
Lemma~\ref{lem:punchline1}
with the roles of $\p$ and $\q$ reversed as well as the roles of
$\pmb{\alpha}$ and $\pmb{\beta}$.

\end{proof}

\section{Reconstruction Algorithm}
\label{sec:recon}
A straightforward iterative application of 
Theorem~\ref{thm:punchlineALG} 
leads to a 
real-computation algorithm for 
reconstructing an unknown generic configuration $\p$ from 
an unlabeled path or loop measurement ensemble
that allows for trilateration.
The algorithm  performs a brute force combinatorial search for 
a trilateration sequence, testing validity along the way using 
rational rank tests and 
Cayley-Menger determinants.  
This generalizes the TRIBOND algorithm from~\cite{dux2}.  
The TRIBOND algorithm only works on edge-length ensembles,
while ours works in the path and loop setting.

Theorem~\ref{thm:punchlineALG}
% (see also Theorem~\ref{thm:consist})
tells us 
that, assuming genericity of $\p$,
we can apply trilateration greedily
without backtracking.
At each step, 
if we can interpret some of our measurement data
as being consistent with a nicely trilaterated point set $\q$, then this
is a certificate of correctness for that step,
up to scale. Meanwhile, if we also assume that
our data arose from an ensemble that 
allows for trilateration (which must be nice by Lemma~\ref{lem:Cdep}), then this means that
by searching, we will be able to find such a certifiably nicely trilaterated 
configuration $\q$, and with the same number of points as $\p$.
Moreover (see below), we can also determine the correct scale.

For the path (resp.~loop) algorithm we have the following specification.

\vspace{0.25 cm}

\noindent
{\bf Input:} dimension, bounce bound, and real-valued
data set, $(d, b, \v)$.\\
\noindent
{\bf Assumption:} Data set $\v$ arises from some generic configuration
$\p$ of $n$ points (for some $n$) in $\RR^d$, under some
$b$-bounded path (resp.~loop) ensemble that allows for trilateration.\\
\noindent
{\bf Output:} A configuration $\q$, where $\q$ is related to $\p$
through a vertex relabeling and a Euclidean congruence.

\vspace{0.25 cm}

\begin{theorem}
\label{thm:alg}
There is an algorithm in a real computational model
that solves the generic path or loop reconstruction problem.
Its running time is polynomial in $|\v|$ and $b$, and exponential in $d$.
\end{theorem}

The algorithm first exhaustively searches for all ordered 
$D$-tuples of $\v$ that describe
$K_{d+2}$ subconfigurations of $\p$. Each of these is treated as a 
``candidate base'': We do not know at first which of these candidates will ultimately 
form the base of a complete trilateration sequence. 
Each candidate base is then grown into an expanding sequence of 
``candidate subconfigurations" by finding new points that connect to an 
existing candidate subconfiguration through a trilaterating set of 
$d+1$ edges or loops. 

To expand a candidate subconfiguration by one point, we search 
exhaustively over all subsets of $d+1$ points in the configuration
and subsets of $d+1$-tuples of values in $\v$ to find a valid 
trilateration step.  
Besides the Cayley-Menger determinant test, 
a valid step is also required to  
have the appropriate
rational rank. 
If the values of $\w$ are not
rationally independent, 
Lemma~\ref{lem:Cdep} implies that 
there is a rational relation
on its underlying functionals.
In the $b$-bounded
setting,  
Lemma~\ref{lem:depC} then implies that the
coefficients describing this relation are bounded integers. 
This bound is polynomial in $b$.
Thus, we can simply do an exhaustive search
for possible integer relations satisfying this 
bound.

If a new point is found, we extend the candidate
subconfiguration and we continue  searching exhaustively over
the extended configuration. If an already-localized point 
is found, we ignore it and continue to search for a trilateration 
step that produces a new point.  If, after trying all the 
possibilities, we do not find a new point, 
the candidate subconfiguration is maximal, 
and we stop searching for this subconfiguration.  
The whole algorithm stops when no candidate subconfiguration 
can be expanded any further. 

Theorem~\ref{thm:consist} only gives us a $K_{d+2}$
configuration up to a scale. 
Fortunately, as we grow a candidate base, 
as in the proof of Lemma~\ref{lem:punchline1},
we will only be able to do so 
using the same shared scale.
Still, since our algorithm uses a number of 
candidate bases, it may
be the case that we reconstruct various subsets of $\p$ at various scales. 
In more detail, there may be a subensemble of 
$\pmb{\alpha}$ 
that is actually of the form of a trilateration sequence,
but with each path (resp.~loop) measured $s$ times per measurement.
During the trilateration, we treat
these as if they were unscaled, which is effectively
equivalent to scaling the length of each underlying edge by $s$. This results in the
reconstruction of a candidate configuration $s\cdot\p$ that is a
scaled up version of the true configuration. 

The existence of at least one correctly scaled reconstruction is
guaranteed by the trilateration assumption on the underlying
$\pmb{\alpha}$.
Thus, at the end of the trilateration process,
we 
identify the true configuration as the reconstructed subset
with the smallest scale among all configurations with
the same maximal number of points. This configuration, $\q$, will be equal (up to vertex relabeling and congruence) to the one true configuration $\p$.

\subsection{Ideas for Efficiency}
There are a few interesting ideas for reducing the cost of the 
the rational rank test.

First of all, testing for rational rank can be done
much more efficiently 
using the basis reduction algorithm for integer relation testing~\cite{ir1,ir2}.
This algorithm works under a real computational model that includes a 
constant-time floor operation.
Roughly speaking, given a  vector $\w$ 
of 
$D$ real numbers and an integer bound $B$, 
the integer relation testing algorithm will
either report that there is no 
``small'' integer linear relation on $\w$ having  integer coefficients
with vector norm less than $B$, or it will provide
a ``medium-sized''
integer relation with vector norm smaller than 
$2^{D/2} B$.
If the algorithm reports that no small relation exists, then
we can conclude that our $\w$ has the necessary rational rank.
If the algorithm returns a 
medium sized relation we know that we do not have the necessary rational rank.
(In fact, we know 
from Lemma~\ref{lem:depC}, that there must also be a small integer relation.)
For fixed $d$, this algorithm has running
time $O(\log(B))$, which for us is $O(\log(b))$.

Secondly,
in the two dimensional setting, we have shown in~\cite{loopsAlg}
that the only linear subspaces of dimension $3$
or higher that are contained in $L_{2,4}$ are those making up  its singular locus.
When a linear map has rank less than $6$, 
then, from Theorem~\ref{thm:linImage},   of $L_{2,n}$ will be a Zariski open subset
of a linear subspace of $\CC^6$. Under the genericity assumption on $\p$, that space will be contained in $L_{2,4}$. 
Thus, if we know or can verify that a non-singular point
$\w$ of $L_{2,4}$
has rational rank at least $3$, then it must
have rational rank $6$. This condition can be verified on any three values of $\w$.

Finally, the rational rank test can  be completely omitted under 
certain assumptions about the measurement process of the 
generic configuration $\p$.
For example, in the loop setting with dimension 3,
suppose we assume that $\pmb{\alpha}$
consists only of pings and triangles, with each
passing through $\p_1$.
This is not an 
unreasonable assumption for our 
signal processing scenario described in Section~\ref{sec:intro}.
Under this assumption, it can be shown that if
$\w$ consists of $10$ distinct values, then it has rational rank $10$, 
and thus no explicit test is needed.
% \begin{question}
% In $3$ dimensions, 
% are there weaker assumptions on $\pmb{\alpha}$
% that allow for a more efficient rational rank test?
% \end{question}

\subsection{Practical Application}\label{sec:accuracy}

As is often the case in computational geometry,
if we implement an algorithm that relies on exact real 
computation with a digital computer, we obtain an 
approximation that does not come with  the same guarantees.
We are currently experimenting with
such methods in ongoing research, and there are recent experiments by others~\cite{nam2020super} in the loop setting with dimension 3 and with $\pmb{\alpha}$ comprising pings and triangles through $\p_1$. But to date, it remains unclear to us what conditions are required for a numerical implementation to succeed with approximate loop lengths.

In the unlabelled edge-only setting, the greedy
TRIBOND algorithm has been extended to the more robust
(but expensive) 
LIGA algorithm~\cite{dux1,dux2}. 
LIGA 
ranks its reconstruction candidates to minimize error. 
Moreover,
LIGA allows for
backtracking during the reconstruction process.  For non-generic 
or approximate input, incorrect
early decisions can often be detected as erroneous using
subsequent data; they can then be undone. 
  Such ideas could prove useful in the path or loop
setting as well.

\appendix

\section{Algebraic Geometry Preliminaries}\label{sec:geometry}

We summarize the needed definitions and facts about
complex algebraic varieties. For more see~\cite{harris}.
%Extra useful facts that we will not formally need will be in 
%brackets.

In this section, $N$ and $D$ will represent arbitrary numbers.

\begin{definition}
A (complex embedded affine) \defn{variety} 
(or \defn{algebraic set}), $V$,
is a (not necessarily strict)
subset of $\CC^N$, for some $N$,
that is defined by the simultaneous
vanishing of a finite set of polynomial equations 
with coefficients in $\CC$
in the 
variables $x_1, x_2, \ldots, x_N$ which are associated with the 
coordinate axes of $\CC^N$.
We say that $V$ is \defn{defined over} $\QQ$ if it can 
be defined by polynomials with coefficients in $\QQ$.

% A variety can be stratified as a union of a finite number of
% complex analytic submanifolds of $\CC^N$.

A finite union of varieties is a variety.
An arbitrary intersection of varieties is a variety.

A  variety $V$ is \defn{reducible} if it is the proper union of two
varieties $V_1$ and $V_2$. 
(Proper means that $V_1$ is not contained in $V_2$ and 
vice-versa.)
Otherwise (assuming it is non-empty) it is called
\defn{irreducible}.
A non-empty variety has a unique decomposition as a finite proper
union of its
maximal irreducible subvarieties called \defn{components}.
(Maximal means that a component cannot be 
 contained in a larger  irreducible 
subvariety of $V$.)

A variety $V$ has a well-defined (maximal) \defn{dimension} $\Dim(V)$, 
which will agree with the largest $D$ for which there
is an open subset of~$V$, in the standard  topology, 
that is a $D$-dimensional complex submanifold of $\CC^N$.

Any strict subvariety $W$ of an 
irreducible variety $V$ must be of strictly lower dimension.

 The \defn{local dimension} $\Dim_\genericpoint(V)$ at a point $\genericpoint$ in a variety $V$
 is the
 dimension of the highest-dimensional irreducible component of $V$ that contains $\genericpoint$.
A point $\genericpoint$ is called
\defn{smooth} in 
a variety $V$ if the dimension of the 
Zariski tangent space (see~\cite{harris})
equals $\Dim_\genericpoint(V)$. Otherwise $\genericpoint$ is called
\defn{singular} in $V$.
The \defn{locus} of singular points of $V$ is denoted $\sing(V)$.
A smooth point has 
an open neighborhood in ~$V$, in the standard  topology, 
that is a $\Dim_\genericpoint(V)$-dimensional complex submanifold of $\CC^N$.

\end{definition}

\begin{remark}
\label{rem:idealtheoretic}
We will say that an ideal in $L[X_1,\ldots,X_N]$
is \defn{defined over} a field $K \subset L$ if it can be generated by polynomials
with coefficients in $K$.

In the present setting, there is no difference between set-theoretic and ideal-theoretic
notions of the field of definition of a variety.
Suppose that $V$ is cut out as $V(J')$ for some ideal $J'$ 
of 
$\CC[X_1,\ldots,X_N]$
that is 
defined over $\QQ$. Since $\CC$ is algebraically closed, we know that $J:=I(V)=\sqrt{J'}$ (Nullstellensatz).
Meanwhile, since $\QQ$ is a perfect field, the radical ideal $\sqrt{J'}$ 
is also defined over $\QQ$ (see~\cite[tag 030V]{stacks-project}).
\end{remark}

\begin{definition}
\label{def:cons}
A \defn{constructible set} $S$ is a set that can be defined using
a finite number of Boolean set operations over a finite number of
varieties.
We say that $S$ is \defn{defined over} $\QQ$ if the varieties can 
be defined by polynomials with coefficients in $\QQ$.

The \defn{Zariski closure} of $S$ is the smallest variety $V$ containing it.
The set $S$ has the same dimension as its Zariski closure $V$.
If $S$ is defined over $\QQ$, then so too is $V$ (see Theorem~\ref{thm:closeDef} below).

 of a variety $V$ of dimension $D$
under a polynomial map $m$ is a constructible set $S$ of dimension
at most $D$.
If $V$ is irreducible, then so 
too is the Zariski closure of $S$.
(We say that $S$
is \defn{irreducible}.)
If $V$ and $m$ are
defined over $\QQ$, then so too is 
$S$~\cite[Theorem 1.22]{basu}.
\end{definition}

Next we define the notion of generic points in a variety. The motivation
is that nothing algebraically special (and which 
is expressible with rational 
coefficients)
is allowed to happen at such points.
Thus, any such algebraic property holding at one generic point must hold at all generic points.
\begin{definition}
\label{def:gen}
A point in an irreducible variety or
constructible set $V$ defined over $\QQ$ is called 
\defn{generic} if its coordinates do not satisfy any algebraic equation 
with coefficients in $\QQ$ besides those that are satisfied by every point
in $V$.  

% The set of generic points has full measure 
% %and are  (standard topology) dense 
% in $V$.

A generic real point in $\RR^N$ as in Definition~\ref{def:genR}
is also a generic point in $\CC^N$,
considered as a variety, as in the current definition.

\end{definition}

\begin{lemma}
\label{lem:genPush}
Let  $C$ and $M$ be irreducible affine varieties,
and $m$ be a polynomial map, all defined over $\QQ$, 
such that the Zariski closure of $m(C)$ is equal to $M$.
If there exists a polynomial $\phi$,
defined over $\QQ$, that does not vanish identically
over $M$ but does vanish at $m(\p)$ for some 
$\p \in C$, then 
there is a polynomial $\psi$,
defined over $\QQ$, that does not vanish identically
over $C$ but does vanish at $\p$.

Thus,
if $\p \in C$ is generic in $C$, 
then $m(\p)$ is generic in $M$.
\end{lemma}
\begin{proof}
The polynomial is simply $\psi(\x) = \phi(m(\x))$.  This $\psi$ is clearly defined over $\QQ$.
The subset of $X \subseteq M$ for which $\phi$
vanishes is proper, by assumption, and a subvariety, by construction.
Thus ,  $m(C)$, cannot be contained in $X$:  Indeed, if it were, then the Zariski 
closure of $m(C)$ would be a subset of $X$, and then not all of $M$.  Hence, there is an 
$\x \in C$ so that $\phi(m(\x))\neq 0$. Thus $\psi$ is not identically zero over $C$.
\end{proof}

\begin{lemma}
\label{lem:genPull}
Let  $L$ and $M$ be irreducible affine varieties of the same dimension,
and $s$ be a polynomial map, all defined over $\QQ$,
 such that the Zariski closure of $s(L)$ equals $M$.
If there exists a polynomial $\psi$,
defined over $\QQ$, that does not vanish identically
over $L$ but does vanish at some $\bl \in L$, then 
there is a polynomial $\phi$,
defined over $\QQ$, that does not vanish identically
over $M$ but does vanish at $s(\bl)$.

Thus,
if $\bl \in L$ is not generic in $L$, then $s(\bl)$ is not generic in $M$.
\end{lemma}
\begin{proof}[Sketch]
Since $L$ is irreducible, the vanishing locus of $\psi$ must be 
of lower dimension. This subvariety
must map under $s$ into to a lower-dimensional subvariety of $M$ (defined
over $\QQ$). This guarantees the existence of an appropriate $\phi$.
\end{proof}

With our notion of generic fixed, we can prove the two
principles of Section~\ref{sec:intro}.

\begin{lemma}
\label{lem:bij}
If $\A$ is a bijective linear map 
on $\CC^N$, then  under $\A$
of a variety $V$ is a variety
of the same dimension.
If $V$ is irreducible, then so too is this image.
\end{lemma}
\begin{proof}
Let $S := \A(V)$. 
Since $\A$ is bijective, then
there is also a map $\A^{-1}$ acting on $\CC^N$,
and $S$ must be the inverse image
of $V$ under this map.
Thus, by pulling back the defining equations of $V$ through
this map, we see that $S$ must be a variety.

The dimension  follows 
from the fact that maps cannot raise dimension, and our map
is invertible.
\end{proof}

\begin{proof}[Proof of Theorem~\ref{thm:prin1}]
Suppose $\E(V)$ does not lie in $W$. Then the preimage
$\E^{-1}(W)$, which is a variety defined over $\QQ$, does not
contain $V$, and the inclusion of $\genericpoint$ in this preimage
would render $\genericpoint$  a non-generic point of $V$.  
\end{proof}

\begin{proof}[Proof of Theorem~\ref{thm:prin2}]
From Theorem~\ref{thm:prin1}, $\A(V) \subseteq V$.
From Lemma~\ref{lem:bij}, $A(V)$ is an algebraic 
subvariety of $V$ of the same dimension, which 
from the assumed irreducibility must be $V$ itself.
\end{proof}

For completeness, we next prove the following theorem.  Its  
proof was suggested by Brian Osserman.
\begin{theorem}
\label{thm:closeDef}
Let $S$ be a constructible set defined over $\QQ$.
Then its Zariski closure is also defined over $\QQ$. 
\end{theorem}

The following material is a bit more technical.
\begin{lemma}\label{lem:ext}
    Let $S$ be a constructible set in $\CC^N$ which is defined over $\QQB$.
    Let $V$ be the Zariski closure of $S$. Let $J$ be the ideal of $V$.
    Let $s$ be the $\QQB$ points of $S$. Let $v$ be the Zariski closure of $s$ in $\QQB^N$.
    Let $j$ be the ideal of $v$.
    Then $J=j^*$, the complex extension of $j$.
\end{lemma}
\begin{proof}
As a constructible set, $S = \bigcup_i (U_i \cap Z_i)$ where $U_i$ is Zariski open and $Z_i$ is Zariski closed.
After splitting into irreducible components, each $Z_i$ can be assumed to be irreducible.
From~\cite[Tag 038I]{stacks-project}, each $Z_i$ is still defined over $\QQB$.
We can also assume that $Z_i \subsetneq U_i$ (otherwise, we can just remove this term from the union).
Thus we can write $S=\bigcup_i (Z_i-X_i)$ where $Z_i$ and $X_i$ are both Zariski closed,
$Z_i$ is irreducible and $Z_i \nsubseteq X_i$.
A finite union operation will commute with complex extensions and so 
henceforth, we can just look at the single component case where
$S=Z-X$. 

Let $I(Z)$ and $I(X)$ be the associated  ideals
in
$\CC[X_1,\ldots,X_N]$, both defined over $\QQB$.
$I(Z)$ is prime and $I(X)$ is radical. Moreover, by assumption, $I(X) \nsubseteq I(Z)$.
Under these assumptions, 
$S$ is a non-empty, Zariski open subset of an irreducible $Z$, so it is dense in $Z$;
the Zariski closure $V$ of $S$ will just be $Z$, with ideal $J=I(Z)$.

Now we restrict $S$, $Z$, and $X$ to $\QQB^N$, where we 
switch to lower case and write
$s=z-x$.
Associated with $z$ and $x$, we have the 
ideals $I(z)$ and $I(x)$ in $\QQB[X_1,\ldots,X_N]$.

Next we show that $I(Z)$ is the complex extension of $I(z)$ (and likewise for $I(X)$):
Since $I(Z)$ is defined over $\QQB$, we can pick a set  of polynomials with coefficients in $\QQB$ that
generate the complex ideal $I(Z)$. These same polynomials define an ideal, $m$
in $\QQB[X_1,\ldots,X_N]$.
The ideal $m$ is still prime (as we are making the field smaller).
Since they share the
same generating polynomials, the complex extension
of $m$ is $I(Z)$.
We now have $z=V(m)$, and since $\QQB$ is closed, from the  Nullstellensatz  we have $I(z)=m$, which proves our claim.
Moreover $z$ must be  irreducible.
A similar argument shows that under complex extension, $I(x)$ 
(which is likewise radical) becomes
$I(X)$. 

Meanwhile, since for  the extensions we have $I(X) \nsubseteq I(Z)$, 
we must also have $I(x) \nsubseteq I(z)$. This gives us $z \nsubseteq x$.
Thus $s$ is a non-empty, open subset of an irreducible $z$, and so it is dense in $z$;
as above, the Zariski closure $v$ of $s$ will just be $z$, with ideal $j=I(z)$.
Since $I(Z)$ is the complex extension of $I(z)$,
we see that $J$ is the complex extension of $j$.
\end{proof}

\begin{lemma}\label{lem:aut}
Let  $s$ be a set of points in $\QQB^N$ which is
invariant under the Galois group
$\operatorname{Aut}(\QQB/\QQ)$.
Then the Zariski closure $v$ of $s$ in $\QQB^N$
is defined over $\QQ$.
\end{lemma}
\begin{proof}
Suppose that $f_1,\ldots,f_n$   cut out $v$. 
For each $i$, consider all the Galois conjugates of $f_i$.
Each $f_i$ has finitely many coefficients, and each of them
lies in a finite-degree extension of $\QQ$.  Hence 
each of the $f_i$ has a finite Galois orbit.  The 
union of these conjugates is a finite collection of polynomials $g_1,\ldots,g_m$.

Each $g_i$ vanishes on $s$, by the assumed Galois invariance
of $s$.  Thus each $g_i$ is in $I(s)$. By the definition of Zariski closure,
we have $I(s)=I(v)$. And so we have each $g_i$ vanishing on $v$.

Now, for each $1\le i\le m$, we define $h_i = e_i(g_1,\ldots,g_m)$, 
where $e_i$ is the $i$th elementary symmetric polynomial.
Each $h_i$ is Galois invariant, since applying a Galois automorphism 
to $h_i$ will  permute the $g_j$, and $e_i$ is symmetric.
As Galois invariant polynomials, the $h_i$ are defined over $\QQ$.

Since all $g_i$ vanish over $v$, we know that $v \subseteq V(\{h_i\})$.
For the other direction, let $p \notin v$. Then one of our  $f_i$ does not vanish at $p$. So there is some specific
$i_0$ so that 
$g_{i_0}(p)\neq 0$. 

Recall that Vieta's formula says that, for a univariate polynomial $P(x)$, if 
\[
  P(x)=  (x - r_1)\cdots (x - r_m) = x^m + a_{m-1}x^{m-1} + \cdots + a_0
\]
then 
\[
    e_i(r_1, \ldots, r_m) = (-1)^{i}a_{m-i}
\]

Getting back to our setting, 
for each $1\le k\le m$, let us fix the numbers 
%$r_k := e_k(g_1(p),...g_m(p))$.
$r_k := g_k(p)$.
Suppose that for all $i$, we had   $e_i(r_1, \ldots, r_n) = 0$, then in Vieta's formula,
we would have $P(x)=x^m$, which would mean that all of the roots 
$(r_1, \ldots, r_n)$
of $P(x)$ were $0$.
But we have at least one $r_{i_0}\neq 0$, a contradiction.
So there must be some  ${i_1}$ so that
 $e_{i_1}(r_1,\ldots,r_m) \neq 0$. 
But this means that 
$
h_{i_1}(p)
=
e_{i_1}(g_1(p),\ldots,g_m(p))
\neq 0$.
Thus $v\supseteq V(\{h_i\})$ 
and so $v = V(\{h_i\})$. 
Since the $h_i$ are defined over $\QQ$, this
establishes the lemma.
\end{proof}
\begin{remark}
The lemma can also be proven in a related but more abstract fashion by first establishing that the Zariski closure, $v$, is Galois 
invariant and then directly appealing to \cite[Tag 038B]{stacks-project}.
\end{remark}

\begin{proof}[Proof of Theorem~\ref{thm:closeDef}]
Let us consider $s$, the $\QQB^N$ restriction of $S$, along with its Zariski closure $v$.
Since $s$ is defined using the vanishing and non-vanishing of polynomials defined over $\QQ$, 
the set $s$ must be invariant to $\operatorname{Aut}(\QQB/\QQ)$. From Lemma~\ref{lem:aut}, its Zariski closure  $v$ in
$\QQB^N$ must be cut out (set theoretically) as $V(j')$ for some ideal $j'$ 
of 
$\QQB[X_1,\ldots,X_N]$
that is 
defined over $\QQ$. 
As in Remark~\ref{rem:idealtheoretic}, since $\QQB$ is closed,
$j:=I(v)$ is equal to $\sqrt{j'}$, its radical ideal in  $\QQB[X_1,\ldots,X_N]$,
which is also defined over $\QQ$.
Meanwhile from Lemma~\ref{lem:ext}, the ideal $J$ of the (complex) Zariski closure of $S$ is just
the complex extension of $j$. Thus $J$ too must be defined over $\QQ$.
\end{proof}

\begin{lemma}
\label{lem:zclose}
Let $s$ be a real semi-algebraic set. 
Let $V:=V(s)$ be its complex Zariski closure.
Then the complex dimension of $V$ equals the real dimension of $s$.
\end{lemma}
\begin{proof}
We say that $s$ is irreducible if its real Zariski closure is irreducible.
If $s$ is not irreducible, we can write it is as a finite proper union
of irreducible semi-algebraic sets $s_k$. We have $\Dim(s) = \max_k(\Dim(s_k))$,
and so we can work over each component separately. Thus, in what follows, we will 
assume that $s$ is irreducible.

Let $i := i(s)$ be the ideal of real polynomials vanishing on $s$, and
$v := v(i(s))$ be the real Zariski closure. (We use lower case $v$ and $i$ to
make it clear that this is all done in the real setting.)
Let $I := I(s)$ be the ideal of complex  polynomials vanishing on $s$, and
$V := V(I(s))$ be the complex Zariski closure. 
As described in 
\cite[Section 2.8]{boch}, the real dimension of $s$
equals the real dimension of $v$.
Meanwhile, from \cite[Proof of Lemma 6]{Whitney},
the ideal $I$ equals $i^*$, the complex extension of $i$.
From~\cite[Proof of Lemma 9]{Whitney},
for any radical real ideal $j$, 
the maximal Jacobian rank 
of $V(j^*)$ must equal
the maximal Jacobian rank of 
of $v(j)$.
Since $v$ is irreducible, then 
from~\cite[Lemma 7]{Whitney}, so too is $V$.
Since $v$ and $V$ are both irreducible, this maximal Jacobian rank
gives us the real (resp. complex) codimensions of $v$ and $V$.
\end{proof}

\section{Rational Functionals and Relations}\label{sec:rational:funcs}
In this section, we prove some generally useful facts about
rational functionals and relations acting on generic point configurations $\p$.

\begin{theorem}
\label{thm:functional}
Let $\bl$ be a generic point in $\LN$ with  $d \ge 2$,
and let $\alpha$ be a rational length functional. 
Suppose $\la \alpha, \bl\ra=0$, then 
$\alpha=0$. Likewise (due to linearity), 
if $\la \alpha,\bl\ra=\la \alpha',\bl\ra$, then $\alpha=\alpha'$.

Similarly, let $\p$ be a generic configuration in $\RR^d$ with $d\geq 2$.
Suppose $\la \alpha, \p \ra=0$, then 
$\alpha=0$. Likewise, if 
$\la \alpha, \p\ra=\la \alpha',\p\ra$, then $\alpha=\alpha'$.
\end{theorem}

Recall from Theorem~\ref{thm:Lvariety} that, assuming $d\geq 2$, 
$\LN$ is irreducible, hence it has generic points.
Additionally, when $\p$ is a  generic  configuration, then 
$l(\p)$ is generic in 
$\LN$.

\begin{proof}
The equation $\la \alpha, \bl\ra=0$ describes an algebraic equation over $\LN$ 
with coefficients in $\QQ$ and that vanishes at $\bl$. 
Genericity of $\bl$ implies that $\alpha$ vanishes
over all of $\LN$.

Suppose, for a contradiction, that $\alpha\neq 0$.
First observe that there are points $\x$ in $\LN$
that have all real and non-zero coordinates.  Because
$\LN$ is symmetric under sign-negation, given 
$\alpha$, we can construct from $\x$ a point 
$\x'$ in $L_{d,n}$ so that the coefficient $\alpha^{ij}$ and 
$\x'_{ij}$ have the same sign for all $ij$
with $\alpha^{ij}\neq 0$.
But then $\langle \alpha,\x'\rangle > 0$, which is a
contradiction.  Hence $\alpha = 0$.

\end{proof}
\begin{remark}
When $d=1$, there 
can be  generic  configurations $\p$ such that
$\la \alpha, \p\ra=0$ with $\alpha \neq 0$.
For example, suppose $n=3$, and let
$\alpha_{12}=1, \alpha_{23}=1, \alpha_{13}=-1$.
Then, $\la \alpha, \p\ra=0$, whenever we have the order $\p_1 \leq \p_2 \leq \p_3$, 
or the reverse order $\p_1 \geq \p_2 \geq \p_3$, 
and
$\la \alpha, \p\ra\neq 0$ otherwise. 
In this case,  $L_{1,3}$ is reducible,
and we have an equation that vanishes identically on one component of the variety
but not on the others.
\end{remark}

The following is useful to tell when a set of rational functionals
is linearly dependent.

\begin{lemma}
\label{lem:depC}
Let $\p$ be a configuration, $\alpha_i$ a sequence of $k$ 
rational functionals, and $v_i := \la \alpha_i,\p \ra$. Suppose that
the functionals $\alpha_i$ 
are linearly dependent. Then, there
is a linear 
dependence that can be expressed as 
$\sum_i c^i \alpha_i=0$, where the
coefficients $c^i$ are rational, not all vanishing.
Moreover, this gives us the 
relation $\sum_i c^i v_i =0$ with the same coefficients.  If the functionals $\alpha_i$ are
integer and $b$-bounded, then we can find such coefficients $c^i$ that are integers, 
bounded in magnitude by $(k')^{k'/2}b^{k'}$.
\end{lemma}
\begin{proof}
Let $k'<k$ be the dimension of the span of the $\alpha_i$.

i) Let us look at the case $k' < \edgecard$.

Pick a subset of the $\alpha_i$ that is minimally linearly dependent
with size 
$k'+1$.
Let us use these as the $k'+1$ rows of a matrix $\M$ with
$\edgecard$ columns.
Each
of its minors of size $k'+1$ must
vanish. 

Pick $k'$ columns that are linearly independent. Append to these, 
one column made up of $k'+1$ variables. The condition that the
determinant of this $(k'+1)\times (k'+1)$ matrix vanishes gives us
a non-trivial linear homogeneous equation in the variables.
In the $b$-bounded setting, Hadamard's bound implies that the 
coefficients of this equation have magnitude at most 
$(k')^{k'/2}b^{k'}$.
As every column of $\M$ is in the span of our chosen $k'$ columns,
the entries in each column of $\M$  must satisfy this equation.
Thus we have found a rational relation on the $k'+1$ rows of $\M$, giving us
a rational relation on the $\alpha_i$.

ii) Let us look at the case $k' = \edgecard$.

Pick a subset of the $\alpha_i$ of size $\edgecard$ that is 
linearly independent.
Let us use these as the rows of a square non-singular matrix $\M$.
Pick one more functional $\beta$ from the $\alpha_i$. Let us think of 
$\beta$ as a row vector of length $\edgecard$.
Since $\beta$ is in the span of our selected rows, we have
$[\beta \, {\rm adj}(\M)] \M = \beta [\det(\M)]$.
Here ``adj'' denotes the adjugate matrix.
This gives us a rational relation between the rows of 
$\M$ and $\beta$, with the coefficients
in brackets above.
Again in the $b$-bounded setting, the coefficients
are bounded in magnitude by $(k')^{k'/2}b^{k'}$.

In both cases i) and ii), the relation on the $v_i$ follows
immediately.
\end{proof}

\begin{lemma}
\label{lem:Cdep}
Let $\p$ be a generic configuration in two or more dimensions.
Let $\alpha_i$ be a  sequence  of $k$ 
rational functionals.
Let $v_i := \la \alpha_i,\p \ra$.
Suppose there is a sequence 
of $k$
rational coefficients $c^i$, not all zero, such that 
$\sum_i c^i v_i=0$.
Then, there is a linear dependence in the functionals $\alpha_i$.
\end{lemma}
\begin{proof}
\ba 
0&=& \sum_i c^i v_i \\
&=& \sum_i c^i \la \alpha_i, \p \ra \\
&=& \la \sum_i c^i \alpha_i, \p \ra 
\ea
Then, from Theorem~\ref{thm:functional}, $\sum_i c^i \alpha_i$ must
be the zero functional.
\end{proof}

\end{document}

%% file: figs/intro_latex.pdf_tex
%% Creator: Inkscape inkscape 0.91, www.inkscape.org
%% Accompanies image file 'intro_latex.pdf' (pdf, eps, ps)
\begingroup%
  \makeatletter%
  \providecommand\color[2][]{%
    \errmessage{(Inkscape) Color is used for the text in Inkscape, but the package 'color.sty' is not loaded}%
    \renewcommand\color[2][]{}%
  }%
  \providecommand\transparent[1]{%
    \errmessage{(Inkscape) Transparency is used (non-zero) for the text in Inkscape, but the package 'transparent.sty' is not loaded}%
    \renewcommand\transparent[1]{}%
  }%
  \providecommand\rotatebox[2]{#2}%
  \ifx\svgwidth\undefined%
    \setlength{\unitlength}{482.0107864bp}%
    \ifx\svgscale\undefined%
      \relax%
    \else%
      \setlength{\unitlength}{\unitlength * \real{\svgscale}}%
    \fi%
  \else%
    \setlength{\unitlength}{\svgwidth}%
  \fi%
  \global\let\svgwidth\undefined%
  \global\let\svgscale\undefined%
  \makeatother%
  \begin{picture}(1,0.27971549)%
    \put(0,0){\includegraphics[width=\unitlength,page=1]{intro_latex.pdf}}%
    \put(0.3236002,0.07845043){\color[rgb]{0,0,0}\makebox(0,0)[lt]{\begin{minipage}{0.24763855\unitlength}\centering $\p_1$\end{minipage}}}%
    \put(0.49916941,0.01972561){\color[rgb]{0,0,0}\makebox(0,0)[lt]{\begin{minipage}{0.24763855\unitlength}\centering $\p_2$\end{minipage}}}%
    \put(0.20911054,0.2832542){\color[rgb]{0,0,0}\makebox(0,0)[lt]{\begin{minipage}{0.24763855\unitlength}\centering $\p_3$\end{minipage}}}%
    \put(0.10512854,0.0184105){\color[rgb]{0,0,0}\makebox(0,0)[lt]{\begin{minipage}{0.24763855\unitlength}\centering $\p_4$\end{minipage}}}%
    \put(-0.09979458,0.16663208){\color[rgb]{0,0,0}\makebox(0,0)[lt]{\begin{minipage}{0.24763855\unitlength}\centering $\p_5$\end{minipage}}}%
    \put(0.64633136,0.22891301){\color[rgb]{0,0,0}\makebox(0,0)[lt]{\begin{minipage}{0.24763855\unitlength}\centering $\p_6$\end{minipage}}}%
    \put(0.85216203,0.0306374){\color[rgb]{0,0,0}\makebox(0,0)[lt]{\begin{minipage}{0.24763855\unitlength}\centering $\p_7$\end{minipage}}}%
  \end{picture}%
\endgroup%

%% file: figs/tetrameasure_latex.pdf_tex
%% Creator: Inkscape inkscape 0.92.1, www.inkscape.org
%% Accompanies image file 'tetrameasure_latex.pdf' (pdf, eps, ps)
\begingroup%
  \makeatletter%
  \providecommand\color[2][]{%
    \errmessage{(Inkscape) Color is used for the text in Inkscape, but the package 'color.sty' is not loaded}%
    \renewcommand\color[2][]{}%
  }%
  \providecommand\transparent[1]{%
    \errmessage{(Inkscape) Transparency is used (non-zero) for the text in Inkscape, but the package 'transparent.sty' is not loaded}%
    \renewcommand\transparent[1]{}%
  }%
  \providecommand\rotatebox[2]{#2}%
  \ifx\svgwidth\undefined%
    \setlength{\unitlength}{639.80015071bp}%
    \ifx\svgscale\undefined%
      \relax%
    \else%
      \setlength{\unitlength}{\unitlength * \real{\svgscale}}%
    \fi%
  \else%
    \setlength{\unitlength}{\svgwidth}%
  \fi%
  \global\let\svgwidth\undefined%
  \global\let\svgscale\undefined%
  \makeatother%
  \begin{picture}(1,0.53651442)%
    \put(0,0){\includegraphics[width=\unitlength,page=1]{tetrameasure_latex.pdf}}%
    \put(0.70537648,0.01260982){\color[rgb]{0,0,0}\makebox(0,0)[lt]{\begin{minipage}{0.18656522\unitlength}\centering $\p_1$\end{minipage}}}%
    \put(0.43848167,0.10431847){\color[rgb]{0,0,0}\makebox(0,0)[lt]{\begin{minipage}{0.18656522\unitlength}\centering $\p_2$\end{minipage}}}%
    \put(0.88862224,0.10431847){\color[rgb]{0,0,0}\makebox(0,0)[lt]{\begin{minipage}{0.18656522\unitlength}\centering $\p_3$\end{minipage}}}%
    \put(0.62902049,0.24518243){\color[rgb]{0,0,0}\makebox(0,0)[lt]{\begin{minipage}{0.18656522\unitlength}\centering $\p_4$\end{minipage}}}%
    \put(0.19060142,0.01260982){\color[rgb]{0,0,0}\makebox(0,0)[lt]{\begin{minipage}{0.18656522\unitlength}\centering $\p_1$\end{minipage}}}%
    \put(-0.07518295,0.10431847){\color[rgb]{0,0,0}\makebox(0,0)[lt]{\begin{minipage}{0.18656522\unitlength}\centering $\p_2$\end{minipage}}}%
    \put(0.37495761,0.10431847){\color[rgb]{0,0,0}\makebox(0,0)[lt]{\begin{minipage}{0.18656522\unitlength}\centering $\p_3$\end{minipage}}}%
    \put(0.11424542,0.24518243){\color[rgb]{0,0,0}\makebox(0,0)[lt]{\begin{minipage}{0.18656522\unitlength}\centering $\p_4$\end{minipage}}}%
    \put(0,0){\includegraphics[width=\unitlength,page=2]{tetrameasure_latex.pdf}}%
    \put(0.19171189,0.29128062){\color[rgb]{0,0,0}\makebox(0,0)[lt]{\begin{minipage}{0.18656522\unitlength}\centering $\p_1$\end{minipage}}}%
    \put(-0.07518295,0.38298931){\color[rgb]{0,0,0}\makebox(0,0)[lt]{\begin{minipage}{0.18656522\unitlength}\centering $\p_2$\end{minipage}}}%
    \put(0.37495761,0.38298931){\color[rgb]{0,0,0}\makebox(0,0)[lt]{\begin{minipage}{0.18656522\unitlength}\centering $\p_3$\end{minipage}}}%
    \put(0.11535592,0.5238533){\color[rgb]{0,0,0}\makebox(0,0)[lt]{\begin{minipage}{0.18656522\unitlength}\centering $\p_4$\end{minipage}}}%
    \put(0,0){\includegraphics[width=\unitlength,page=3]{tetrameasure_latex.pdf}}%
    \put(0.70537652,0.29128067){\color[rgb]{0,0,0}\makebox(0,0)[lt]{\begin{minipage}{0.18656522\unitlength}\centering $\p_1$\end{minipage}}}%
    \put(0.43848169,0.38298931){\color[rgb]{0,0,0}\makebox(0,0)[lt]{\begin{minipage}{0.18656522\unitlength}\centering $\p_2$\end{minipage}}}%
    \put(0.88862224,0.38298931){\color[rgb]{0,0,0}\makebox(0,0)[lt]{\begin{minipage}{0.18656522\unitlength}\centering $\p_3$\end{minipage}}}%
    \put(0.62902055,0.5238533){\color[rgb]{0,0,0}\makebox(0,0)[lt]{\begin{minipage}{0.18656522\unitlength}\centering $\p_4$\end{minipage}}}%
    \put(0,0){\includegraphics[width=\unitlength,page=4]{tetrameasure_latex.pdf}}%
  \end{picture}%
\endgroup%

%% file: figs/1dambiguity_latex.pdf_tex
%% Creator: Inkscape inkscape 0.91, www.inkscape.org
%% Accompanies image file '1dambiguity_latex.pdf' (pdf, eps, ps)
\begingroup%
  \makeatletter%
  \providecommand\color[2][]{%
    \errmessage{(Inkscape) Color is used for the text in Inkscape, but the package 'color.sty' is not loaded}%
    \renewcommand\color[2][]{}%
  }%
  \providecommand\transparent[1]{%
    \errmessage{(Inkscape) Transparency is used (non-zero) for the text in Inkscape, but the package 'transparent.sty' is not loaded}%
    \renewcommand\transparent[1]{}%
  }%
  \providecommand\rotatebox[2]{#2}%
  \ifx\svgwidth\undefined%
    \setlength{\unitlength}{324.42626953bp}%
    \ifx\svgscale\undefined%
      \relax%
    \else%
      \setlength{\unitlength}{\unitlength * \real{\svgscale}}%
    \fi%
  \else%
    \setlength{\unitlength}{\svgwidth}%
  \fi%
  \global\let\svgwidth\undefined%
  \global\let\svgscale\undefined%
  \makeatother%
  \begin{picture}(1,0.22231539)%
    \put(0,0){\includegraphics[width=\unitlength,page=1]{1dambiguity_latex.pdf}}%
    \put(-0.0655879,0.22757297){\color[rgb]{0,0,0}\makebox(0,0)[lt]{\begin{minipage}{0.36792475\unitlength}\centering $\p_1$\end{minipage}}}%
    \put(0.7588935,0.22757297){\color[rgb]{0,0,0}\makebox(0,0)[lt]{\begin{minipage}{0.36792475\unitlength}\centering $\p_3$\end{minipage}}}%
    \put(0.46316347,0.22757297){\color[rgb]{0,0,0}\makebox(0,0)[lt]{\begin{minipage}{0.36792475\unitlength}\centering $\p_2$\end{minipage}}}%
    \put(0.46316347,0.02735308){\color[rgb]{0,0,0}\makebox(0,0)[lt]{\begin{minipage}{0.36792475\unitlength}\centering $\q_2$\end{minipage}}}%
    \put(0.7588935,0.02735308){\color[rgb]{0,0,0}\makebox(0,0)[lt]{\begin{minipage}{0.36792475\unitlength}\centering $\q_3$\end{minipage}}}%
    \put(0.34663734,0.02735308){\color[rgb]{0,0,0}\makebox(0,0)[lt]{\begin{minipage}{0.36792475\unitlength}\centering $\q_1$\end{minipage}}}%
    \put(-0.11370999,0.10923348){\color[rgb]{0,0,0}\makebox(0,0)[lt]{\begin{minipage}{0.36792475\unitlength}\centering $-\lambda$\end{minipage}}}%
    \put(0.30209853,0.10923348){\color[rgb]{0,0,0}\makebox(0,0)[lt]{\begin{minipage}{0.36792475\unitlength}\centering $0$\end{minipage}}}%
    \put(0.74259161,0.10923348){\color[rgb]{0,0,0}\makebox(0,0)[lt]{\begin{minipage}{0.36792475\unitlength}\centering $\lambda$\end{minipage}}}%
  \end{picture}%
\endgroup%

%% file: figs/2flips_latex.pdf_tex
%% Creator: Inkscape inkscape 0.92.1, www.inkscape.org
%% Accompanies image file '2flips_latex.pdf' (pdf, eps, ps)
\begingroup%
  \makeatletter%
  \providecommand\color[2][]{%
    \errmessage{(Inkscape) Color is used for the text in Inkscape, but the package 'color.sty' is not loaded}%
    \renewcommand\color[2][]{}%
  }%
  \providecommand\transparent[1]{%
    \errmessage{(Inkscape) Transparency is used (non-zero) for the text in Inkscape, but the package 'transparent.sty' is not loaded}%
    \renewcommand\transparent[1]{}%
  }%
  \providecommand\rotatebox[2]{#2}%
  \ifx\svgwidth\undefined%
    \setlength{\unitlength}{324.42626953bp}%
    \ifx\svgscale\undefined%
      \relax%
    \else%
      \setlength{\unitlength}{\unitlength * \real{\svgscale}}%
    \fi%
  \else%
    \setlength{\unitlength}{\svgwidth}%
  \fi%
  \global\let\svgwidth\undefined%
  \global\let\svgscale\undefined%
  \makeatother%
  \begin{picture}(1,0.22231539)%
    \put(0,0){\includegraphics[width=\unitlength,page=1]{2flips_latex.pdf}}%
    \put(-0.06558794,0.22757297){\color[rgb]{0,0,0}\makebox(0,0)[lt]{\begin{minipage}{0.36792476\unitlength}\centering $\p_1$\end{minipage}}}%
    \put(0.75889349,0.22757297){\color[rgb]{0,0,0}\makebox(0,0)[lt]{\begin{minipage}{0.36792476\unitlength}\centering $\p_3$\end{minipage}}}%
    \put(0.20728265,0.22757334){\color[rgb]{0,0,0}\makebox(0,0)[lt]{\begin{minipage}{0.36792476\unitlength}\centering $\p_2$\end{minipage}}}%
    \put(0.20728265,0.02735308){\color[rgb]{0,0,0}\makebox(0,0)[lt]{\begin{minipage}{0.36792476\unitlength}\centering $\q_2$\end{minipage}}}%
    \put(0.75889349,0.02735308){\color[rgb]{0,0,0}\makebox(0,0)[lt]{\begin{minipage}{0.36792476\unitlength}\centering $\q_3$\end{minipage}}}%
    \put(-0.06558794,0.02735308){\color[rgb]{0,0,0}\makebox(0,0)[lt]{\begin{minipage}{0.36792476\unitlength}\centering $\q_1$\end{minipage}}}%
    \put(0.32032835,0.22757334){\color[rgb]{0,0,0}\makebox(0,0)[lt]{\begin{minipage}{0.36792476\unitlength}\centering $\p_4$\end{minipage}}}%
    \put(0.64604305,0.02735308){\color[rgb]{0,0,0}\makebox(0,0)[lt]{\begin{minipage}{0.36792476\unitlength}\centering $\q_4$\end{minipage}}}%
  \end{picture}%
\endgroup%